\documentclass[onefignum,onetabnum]{siamart190516}

\usepackage{lipsum}
\usepackage{amsfonts}
\usepackage{graphicx}
\usepackage{epstopdf}
\usepackage{algorithmic}
\usepackage{ntheorem}
\usepackage{verbatim}

\newsiamremark{remark}{Remark}
\newsiamremark{hypothesis}{Hypothesis}
\crefname{hypothesis}{Hypothesis}{Hypotheses}
\newsiamthm{claim}{Claim}

\definecolor{darkred}{rgb}{0.85,0,0}

\definecolor{green}{rgb}{0,0.7,0}

\graphicspath{{./pics/}}

\newtheorem{example}{Example}[section]
\newtheorem{assumption}{Assumption}[section]

\newcommand{\al}{\alpha}
\newcommand{\fy}{\varphi}
\renewcommand{\d}{{\rm d}}
\def\Dal{{\partial_t^\al}}

\def\Om{\Omega}
\def\II{(\Om)}

\def\dH#1{\dot H^{#1}(\Omega)}
\def\Uad{\mathcal{A}}

\begin{document}

\title{Numerical Estimation of a Diffusion Coefficient in Subdiffusion\thanks{The work of B. Jin is supported by UK EPSRC grant EP/T000864/1, and
 the research of Z. Zhou is supported by Hong Kong RGC grant (No. 15304420).} }

\author{Bangti Jin\thanks{Department of Computer Science, University College London, Gower Street, London, WC1E 6BT, UK
(\texttt{b.jin@ucl.ac.uk, bangti.jin@gmail.com})}
\and Zhi Zhou\thanks{Department of Applied Mathematics, The Hong Kong Polytechnic University, Kowloon, Hong Kong.
(\texttt{zhizhou@polyu.edu.hk, zhizhou0125@gmail.com})}}

\maketitle

\begin{abstract}
In this work, we consider the numerical recovery of a spatially dependent diffusion coefficient in a subdiffusion model from distributed observations.
The subdiffusion model involves a Caputo fractional derivative of order $\alpha\in(0,1)$ in time. The numerical estimation is based
on the regularized output least-squares formulation, with an $H^1(\Omega)$ penalty. We prove the well-posedness of the
continuous formulation, e.g., existence and stability. Next, we develop a fully discrete scheme based on the Galerkin finite element method in
space and backward Euler convolution quadrature in time. We prove the subsequential convergence of the sequence of discrete
solutions to a solution of the continuous problem as the discretization parameters (mesh size and time step size) tend to zero.
Further, under an additional regularity condition on the exact coefficient, we derive convergence rates in a weighted $L^2(\Omega)$ norm for
the discrete approximations to the exact coefficient {in the one- and two-dimensional cases}. The analysis relies heavily on suitable nonstandard nonsmooth data error estimates for the
direct problem. We provide illustrative numerical results to support the theoretical study.
\end{abstract}

\begin{keywords}
parameter identification, subdiffusion, fully discrete scheme, convergence, error estimate, Tikhonov regularization
\end{keywords}

\begin{AMS}
35R11, 35R30, 49M25, 65M60 
\end{AMS}

\section{Introduction}\label{sec:intro}
Let $\Omega\subset\mathbb{R}^d$ ($d=1,2,3$) be a convex polyhedral domain with a boundary
$\partial\Omega$. Consider the following initial-boundary value problem of the subdiffusion equation:
\begin{equation}\label{eqn:fde}
    \left\{\begin{aligned} \Dal  u(x,t) -\nabla\cdot(q(x)\nabla u(x,t))  &= f(x,t),&&  (x,t)\in\Omega\times(0,T],\\
    u(x,0) &= u_0(x), && x\in \Omega,\\
u(x,t) &=0,&& (x,t)\in\partial\Omega\times(0,T],
   \end{aligned} \right.
\end{equation}
where $T>0$ is the final time. The functions $f$ and $u_0$ are the given source term and initial condition, respectively, and their
precise regularity will be specified below. The notation $\Dal u$,
denotes the Caputo fractional derivative in time of order $\alpha\in(0,1)$, defined by \cite{KilbasSrivastavaTrujillo:2006}
\begin{equation*}
  \Dal u (t) = \frac{1}{\Gamma(1-\alpha)} \int_0^t(t-s)^{-\alpha}  {u'(s)}\ {\rm d}s.
\end{equation*}
The fractional derivative $\partial_t^\alpha u$ recovers the usual first order derivative $u'(s)$ as the order $\alpha\to1^-$
for smooth functions $u$. Thus the model is a fractional analogue of the classical diffusion model. {Throughout,
we denote the solution to problem \eqref{eqn:fde} by $u(q)$ to explicitly indicate its dependence on the diffusion coefficient $q$.}

The model \eqref{eqn:fde} has received enormous attention in recent years, due to their extraordinary capability for
describing anomalously slow diffusion processes (also known as subdiffusion), which displays local motion occasionally
interrupted by long sojourns and trapping effects.  At a microscopic level, such anomalous diffusion processes are
accurately modeled by continuous time random walk, where the waiting time between consecutive particle jumps
follows a heavy tailed distribution with a divergent mean, and the model \eqref{eqn:fde} is the macroscopic counterpart,
describing the evolution of the probability density function (in $\mathbb{R}^d$) of the particle appearing at time
$t$ and spatical location $t$, in analogy to Brownian motion for
normal diffusion. These processes are characterized by sublinear growth of the particle mean squared displace with
the time. The model \eqref{eqn:fde} has found many applications in physics, biology and finance etc, including electron
transport with copier \cite{ScherMontroll:1975}, thermal diffusion on fractal domains \cite{Nigmatulli:1986},
dispersive transport of ions in column experiments \cite{AdamsGelhar:1992,HatanoHatano:1998}, protein transport within
membranes \cite{KouXie:2004,Kou:2008,RitchieShan:2005}, and solute transport in heterogeneous media \cite{DentzCortis:2004,
BerkowitzCortis:2006}. We refer interested readers to the comprehensive reviews \cite{BouchaudGeorges:1990,MetzlerKlafter:2000,
MetzlerJeon:2014} for physical modeling and long lists of diverse applications.

This work is concerned with numerically identifying the diffusion coefficient $q^\dag\in L^\infty (\Omega)$ the  model
\eqref{eqn:fde} from the (noisy) distributed observation
\begin{equation}\label{eqn:observ}
z^\delta (x,t) = u(q^\dag)(x,t) + \xi(x,t),\quad (x,t)\in\Omega\times[0,T],
\end{equation}
where $u(q^\dag)$ is the exact data {(corresponding to the exact diffusion coefficient $q^\dag$),}
and $\xi$ denotes the noise, with an accuracy  $\delta=\|z^*-z^\delta\|_{L^2(0,T;L^2(\Omega))}$.
The inverse problem is a fractional analogue of the inverse conductivity problem for standard parabolic problems, which
has been extensively studied both numerically and theoretically; see the monograph \cite[Chapter 9]{Isakov:2006} for
relevant mathematical theory and the references \cite{Gutman:1990,Karkkainen:1997,KeungZou:1998,
EnglZou:2000,DuChateauThewellButters:2004,NilssenTai:2005,WangZou:2010} for a rather incomplete list of works
on numerical identification of a diffusion coefficient in standard parabolic problems. Most of these existing works
formulate the inverse problem into an output least-squares formulation, with a suitable penalty, e.g., Sobolev smoothness
or total variation. Formally, the inverse problem is over-determined for uniqueness / identifiability, and the terminal data at time $T$
or lateral Cauchy data may suffice unique recovery (see \cite[Chapter 9]{Isakov:2006} for relevant uniqueness results for
in standard parabolic case). Nonetheless, numerically, the full space-time datum \eqref{eqn:observ}
or a restricted version over the region  $\Omega\times[T_0,T]$ is frequently employed in existing studies for standard
parabolic problems \cite{Gutman:1990,TaiKarkkarinen:1995,Karkkainen:1997,KeungZou:1998,EnglZou:2000,
DuChateauThewellButters:2004,NilssenTai:2005,WangZou:2010}, due to, e.g., weak regularity assumption on problem data.
In particular, all existing works \cite{KeungZou:1998,NilssenTai:2005,WangZou:2010} on error estimates (for parabolic
problems) require the full space-time datum \eqref{eqn:observ}, and it appears to be open to extend these results to
partial data. Thus, we shall focus the analysis on the full datum \eqref{eqn:observ} below.

In this work, we shall develop a numerical procedure for recovering a spatially dependent diffusion coefficient.
We formulate an output least-squares formulation with an $H^1(\Omega)$ penalty, {which is suitable for recovering
a smooth coefficient $q$}, and provide a complete analysis of both continuous and discrete formulations, including
well-posedness and convergence of discrete approximations, for weak regularity assumption on the problem data, in
\cref{sec:cont,sec:fully}, respectively. Furthermore, in \cref{sec:err}, we derive some \textit{a priori} weighted
$L^2(\Omega)$ error estimates on the discrete approximation under a mild regularity assumption on the exact diffusion
coefficient $q^\dag$ {in one- and two-dimensional cases;} see \cref{thm:error-q} and \cref{cor:err-q}.
The obtained estimates depend on the spatial mesh size $h$,
temporal step size $\tau$, the noise level $\delta$, and the regularization parameter $\gamma$. These results extend the
corresponding results for the standard parabolic case \cite{Gutman:1990,KeungZou:1998,WangZou:2010}, and represent the
main theoretical achievements of the work.

Generally, when compared with standard parabolic problems, the presence of the time-fractional derivative
$\partial_t^\alpha u$ in the model \eqref{eqn:fde} poses a number of distinct challenges to the mathematical and numerical
analysis (see \cite{JinLazarovZhou:2019} for a concise overview): (i) due to the nonlocality of the Caputo derivative $\partial_t^\alpha u$, many
powerful tools from PDE theory and classical numerical analysis, e.g., energy argument and integration by parts formula,
are not directly applicable; (ii) the solution $u$ generally has only limited spatial and temporal regularity, even for
smooth problem data; (iii) high-order time stepping schemes often lack robustness with respect to the regularity of the
problem data; (iv) the nonlocality incurs a storage issue for time-stepping. Naturally, these challenges persist for the analysis of
the regularized output least-squares formulation \cref{eqn:ob}--\cref{eqn:var} below, and especially items (i) and (ii) represent the main technical
challenges in the convergence analysis, and hence it differs substantially from the standard parabolic counterpart.
Further, the error analysis is greatly complicated by the nonlinearity of the forward map, and thus standard techniques from
optimal control theory, via, e.g., convexity and the first-order optimality condition, also do not apply directly.
To overcome these technical challenges, we shall employ the positivity of the fractional
derivative operators (in \cref{lem:a-priori,lem:disc-pos}), nonsmooth data estimates (in \cref{lem:err-1}) and novel test
function $\fy$ (in \cref{thm:error-q}), which represent the main technical novelties of the work.

Now we briefly review relevant works from the inverse problem literature. Inverse problems for fractional diffusion
has started to attract much interest, and there has already been a vast literature (see, e.g., the review
\cite{JinRundell:2015}). There are a number of interesting works on recovering
the diffusion coefficient \cite{ChengYamamoto:2009,LiGuJia:2012,LiYamamoto:2013,Zhang:2016,KianOksanenYamamoto:2018}.
In an influential piece of work, Cheng et al \cite{ChengYamamoto:2009} proved the unique recovery of both diffusion
coefficient and fractional order from the lateral Cauchy data for the model \eqref{eqn:fde} with a Dirac source in the one
spatial dimensional case. The proof employs Laplace transform and Sturm-Liouville theory. Recently,
Kian et al \cite{KianOksanenYamamoto:2018} proved uniqueness for the recovery of two coefficients from the
Dirichlet-to-Neumann map \cite{KianOksanenYamamoto:2018}. Li et al \cite{LiGuJia:2012,LiYamamoto:2013} discussed the
numerical recovery of the diffusion coefficient (simultaneously with the fractional order), and showed various continuity
results of the parameter to state map. However, the numerical discretization was not analyzed in \cite{LiYamamoto:2013}.
Zhang \cite{Zhang:2016} proved the unique recovery for the case of a time-dependent $q\equiv
q(t)$, and devised a numerical scheme for its recovery. See also the work \cite{WeiLi:2018} for further numerical results
on recovering the diffusion coefficient from boundary data in the one-dimensional case, using a space-time variational
formulation, which allows only a zero initial condition. However, there is neither analysis of the discretized problem
nor error estimates in these interesting existing works. In sum, there is no rigorously study of the
discretization schemes, and it is precisely this gap that this work aims to fill
in. We refer interested readers also to the works \cite{ZhangZhou:2017,JiangLiLiuYamamoto:2017,KaltenbacherRundell:2019}
and references therein for further numerical methods on other nonlinear inverse problems for the subdiffusion model.

The rest of the paper is organized as follows. In \cref{sec:cont}, we formulate the continuous problem, and
analyze its well-posedness, e.g., existence and stability. Then in \cref{sec:fully}, we describe a fully discrete
scheme, and show the convergence of the discrete approximations to a solution of the continuous problem as the
discretization parameters tend to zero. In  \cref{sec:err}, we provide detailed error estimates for the discrete
approximations. Finally, in \cref{sec:numer}, we present illustrative one- and two-dimensional numerical results
to complement the analysis. We conclude the paper with further remarks in \cref{sec:concl}.

We end this section with some useful notation. Throughout, the notation $c$ denotes a generic constant, which may change
at each occurrence, but it is always independent of $q$, mesh size $h$ and time stepsize $\tau$. We shall
employ standard notation for Sobolev spaces \cite{AdamsFournier:2003}. The spaces $L^p(\Omega)$ and $W^{k,p}(\Omega)$
are endowed with the norms $\|\cdot\|_{L^p(\Omega)}$ and $\|\cdot\|_{W^{k,p}(\Omega)}$, respectively, and the notation
$(\cdot,\cdot)$ denotes the $L^2(\Omega)$ inner product. We denote by $H^{-1}
(\Omega)$ the dual space of $H_0^1(\Omega)$. For a Banach space $B$ (endowed with the norm $\|\cdot\|_B$), we define
$L^2(0,T;B) = \{u(t)\in B\ \mbox{for a.e. } t\in(0,T)\ \mbox{and } \|u\|_{L^2(0,T;B)}<\infty\}$,
and the norm is given by $\|u\|_{L^2(0,T;B)}=(\int_0^T\|u(t)\|_B^2\d t)^\frac{1}{2}$. The notation $(\cdot,\cdot)_{L^2(0,T;L^2(\Omega))}$
denotes the inner product in the space $L^2(0,T;L^2(\Omega))$.
Similarly, the space $H^1(0,T;B)$ denotes
$H^1(0,T;B)=\{u\in L^2(0,T;B): u'(t)\in L^2(0,T;B)\}$,
with its norm given by $\|u\|_{H^1(0,T;B)} = (\|u\|_{L^2(0,T;B)}^2 + \|u'(t)\|_{L^2(0,T;B)}^2)^\frac{1}{2}$,
{with the notation $'$ denoting the (weak) temporal derivative}. Further,
for any $s\geq 0$, we denote by $\dot H^s(\Omega)=\{v\in L^2(\Omega): {(-\Delta)^\frac{s}{2}v}\in L^2(\Omega)
\}$, where $\Delta$ being the Laplacian with a zero Dirichlet boundary condition and the fractional power $(-\Delta)^\frac{s}{2}$
is defined by the spectral decomposition \cite{Kato:1961}. The space $\dot H^s(\Omega)$ is equipped with the norm $\|v\|_{\dot H^s(\Omega)}=(\|v\|^2_{L^2(\Omega)}+\|{(-\Delta)^\frac{s}{2}}
v\|_{L^2(\Omega)}^2)^\frac{1}{2}$. Then $\dot H^0(\Omega)=L^2(\Omega)$, $\dot H^1(\Omega)=H_0^1(\Omega)$ and $\dot H^2(\Omega)=H^2(\Omega)
\cap H_0^1(\Omega)$.

\section{Well-posedness of the continuous problem}\label{sec:cont}
In this section, we formulate and analyze the continuous formulation of the reconstruction approach. To recover the
diffusion coefficient $q$, we employ the following output least-squares formulation with an $H^1(\Omega)$-penalty:
\begin{equation}\label{eqn:ob}
    \min_{q \in \Uad} J_\gamma(q;z^\delta)=\tfrac12   \|u( q) - z^\delta\|_{L^2(0,T;L^2(\Omega))}^2  + \tfrac\gamma2\|\nabla q \|_{L^2(\Omega) }^2,
\end{equation}
where $u(q)$ satisfies the variational problem
\begin{equation}\label{eqn:var}
   (\Dal u(q),v)+(q\nabla u(q),\nabla v) = (f,v),\quad \forall  v\in\dH1,t\in(0,T], \quad \mbox{with }\quad u(0)=u_0.
\end{equation}
The admissible set $\Uad$ for the diffusion coefficient $q(x)$ is given by
\begin{equation*}
    \Uad=\{ q\in H^1(\Omega):~~c_0\le q \le c_1 ~~\text{a.e. in}~~ \Omega\},
\end{equation*}
with constants $c_0,c_1\in\mathbb{R}$ and $0<c_0< c_1$. The $H^1(\Omega)$ seminorm penalty is suitable for recovering a
smooth diffusion coefficient. {In case of nonsmooth coefficients, alternative penalties, e.g., total variation,
should be employed; see Remark \ref{rmk:alternative} for further discussions.} The scalar $\gamma>0$ is the
regularization parameter, controlling the strength of the penalty \cite{ItoJin:2015}.
The dependence of the functional $J_\gamma$ on $z^\delta$ will be suppressed below whenever there is
no confusion. For the analysis in Sections \ref{sec:cont} and \ref{sec:fully}, we make the following assumption on
problem data. It is sufficient to ensure the existence of a unique solution $u(q)\in L^2(0,T;
H^1(\Omega))$ for any $q\in \Uad$ \cite{JinLazarovZhou:2019}.
\begin{assumption}\label{ass:data1}
$u_0\in {\dot H^1(\Omega)} $, $f\in L^2(0,T;L^2
(\Omega))$, and $z^\delta\in L^2(0,T;L^2(\Omega))$.
\end{assumption}

First we show the well-posedness of problem \eqref{eqn:ob}--\eqref{eqn:var}, which relies on a
continuity result for the parameter-to-state map $u(q)$. First, we recall a stability result on the
solution operator. Below, for any $q\in \Uad$, the operator $A(q):  {\dH1}\rightarrow H^{-1}(\Omega)$ is defined by
\begin{equation*}
  -\langle A(q) \fy, \psi \rangle = (q\nabla \fy, \nabla  \psi), \quad \forall~ {\fy,\psi\in  \dH1},
\end{equation*}
where $\langle\cdot,\cdot\rangle$ denotes the duality pairing between $H^{-1}(\Omega)$ and $\dH1$.
For any $\fy \in {\dot{H}^2(\Omega)}$, there holds $ A(q) \fy =  \nabla\cdot(q\nabla \fy) \in L^2(\Omega)$.
\begin{lemma}\label{lem:a-priori}
For any $q\in \Uad$, let $v$ solve
\begin{align*}
    \partial_t^\alpha v - A(q)  v = f,\quad \forall t\in (0,T],\quad \mbox{with } v(0) = 0.
\end{align*}
Then there holds
\begin{equation*}
  \|v\|_{L^2(0,T;H^1(\Omega))}^2 \leq c\|f\|_{L^2(0,T;H^{-1}(\Omega))}^2.
\end{equation*}
\end{lemma}
\begin{proof}
Taking $\phi=v$ in the weak formulation, and then integrating from $0$ to $T$ give
$$(\partial_t^\alpha v(t),v(t))_{L^2(0,T;L^2(\Omega))} + (q\nabla v,\nabla v)_{L^2(0,T;L^2(\Omega))}= (f,v)_{L^2(0,T;L^2(\Omega))}.$$
Since $v(0)= 0$, {the Caputo fractional derivative coincides with the Riemann-Liouville one \cite{KilbasSrivastavaTrujillo:2006},
and upon extending $v$ by $0$ outside $[0,T]$, it follows directly from \cite[Lemma 2.3]{LiXu:2009}}
that
\begin{equation*}
  (\partial_t^\alpha v(t),v(t))_{L^2(0,T;L^2(\Omega))}\geq 0,
\end{equation*}
and by Poincar\'{e}'s inequality and Cauchy-Schwarz inequality, we obtain the desired estimate.
\end{proof}

The next result gives the continuity of the parameter-to-state map.
\begin{lemma}\label{lem:conti-p2s}
If the sequence $\{q^n\}\subset \Uad$ converges to $q\in \Uad$ almost everywhere, then
\begin{equation*}
  \lim_{n\to\infty} \|u(q)-u(q^n)\|_{L^2(0,T;H^1(\Omega))} = 0.
\end{equation*}
\end{lemma}
\begin{proof}
Let $v^n = u(q)-u(q^n)$. Then it satisfies $v^n(0)=0$ and
\begin{equation*}
  \partial_t^\alpha v^n - \nabla\cdot (q^n\nabla v^n) = \nabla\cdot((q-q^n)\nabla u(q)),\quad \forall t\in(0,T].
\end{equation*}
Then by  \cref{lem:a-priori} and the definition of the $H^{-1}(\Omega)$, we obtain
\begin{align*}
  \|v^n\|_{L^2(0,T;H^1(\Omega))} & \leq c\|\nabla\cdot((q-q^n)\nabla u(q))\|_{L^2(0,T;H^{-1}(\Omega))}\\
    & \leq c\|(q-q^n)\nabla u(q)\|_{L^2(0,T;L^2(\Omega))}.
\end{align*}
{Let $\phi^n = |q-q^n|^2 \int_0^T |\nabla u(q)|^2 \,\d t$, then $\phi^n\rightarrow0$ almost everywhere (a.e.),
since $q^n\to q$ a.e., and further, since $q, q^n \in \mathcal{A}$, we have
$0\leq \phi^n \le 4c_1^2   \int_0^T |\nabla u(q)|^2 \,\d t \in L^1(\Omega).$ Then, Lebesgue's dominated
convergence theorem \cite[Theorem 1.9]{EvansGariepy:2015} implies
\begin{equation*}
\lim_{n\rightarrow\infty} \|(q-q^n)\nabla u(q)\|_{L^2(0,T;L^2(\Omega))}^2 = \lim_{n\rightarrow\infty} \int_\Omega \phi^n(x)\,\d x
 = \int_\Omega  \lim_{n\rightarrow\infty} \phi^n(x)\,\d x = 0,
\end{equation*}}
which shows the desired estimate.
\end{proof}

The next result gives the existence of a minimizer. {With Lemma \ref{lem:conti-p2s} at hand,
the result follows by a standard compactness argument in calculus of variation (see, e.g., \cite{EnglHankeNeubauer:1996,ItoJin:2015}),
and the proof is included only for completeness.}
\begin{theorem}\label{thm:ex}
Under \cref{ass:data1}, there exists at least one minimizer to \eqref{eqn:ob}--\eqref{eqn:var}.
\end{theorem}
\begin{proof}
Since the functional $J_\gamma$ is bounded from below by zero, there exists a minimizing sequence $\{q^n\}_{n\geq1}\subset\Uad$
such that $\lim_{n\to\infty} J_\gamma(q^n)=\inf_{q\in \Uad}J_\gamma(q)$. Thus, the
sequence $\{q^n\}_{n\geq1}$ is uniformly bounded in the $H^1(\Omega)$ seminorm, which together with the box constraint
$q^n\in \Uad $, implies that it is also uniformly bounded in $H^1(\Omega)$. Thus there exists
a subsequence, still denoted by $\{q^n\}_{n\geq 1}$ that converges to some $q^*\in \Uad$ weakly in $H^1(\Omega)$,
and by compact Sobolev embedding theorem \cite{EvansGariepy:2015}, converges also in $L^1(\Omega)$. Further, by standard measure
theory, convergence in $L^1(\Omega)$ implies almost everywhere convergence up to a subsequence \cite[Theorem 1.21, p.
29]{EvansGariepy:2015}. Thus, we may assume that the subsequence $\{q^n\}_{n\geq1}$ converges to $q^*$ in $L^1(\Omega)$
and almost everywhere. {Then by \cref{lem:conti-p2s}, for the sequence $\{u(q^n)\}_{n\geq 1}$ of solutions to
problem \eqref{eqn:fde}, there holds $\lim_{n\to\infty}\|u(q^n)-u(q^*)\|_{L^2(0,T;H^1(\Omega))}=0$.} Then by
Sobolev embedding \cite{AdamsFournier:2003}, $\lim_{n\to\infty} \|u(q^n)-z^\delta\|_{L^2(0,T;L^2(\Omega))}^2 = \|u(q^*)-z^\delta\|_{L^2(0,T;L^2(\Omega))}^2.$
This and weak lower semi-continuity of semi-norms imply that $q^*$ is a minimizer to \eqref{eqn:ob}.
\end{proof}

The following continuous dependence results hold, where the minimum $H^1(\Omega)$-seminorm solution refers to
the solution $q^\dag$ of the minimum $H^1(\Omega)$ seminorm among all solutions within the admissible set $\mathcal{A}$
corresponding to the exact data $z^\dag=u(q^\dag)$. The proof follows by a standard argument (see, e.g.,
\cite{EnglHankeNeubauer:1996,ItoJin:2015}), and thus is omitted.
\begin{theorem}\label{thm:stability}
Under \cref{ass:data1}, the following statements hold.
\begin{itemize}
  \item[$\rm(i)$] Let $\gamma>0$ be fixed, the sequence $\{z_j\}_{j\geq 1}$ be convergent to some data $z$ in $L^2(0,T;L^2(\Omega))$, and $q_j^*\in\Uad$
  the corresponding minimizer to $J_\gamma(\cdot;z_j)$. Then $\{q_j^*\}_{j\geq1}$
  contains a subsequence convergent to a minimizer of $J_\gamma(\cdot;z)$ over $\mathcal{A}$ in $H^1(\Omega)$.
  \item[$\rm(ii)$] Let $\{\delta_j\}_{j\geq1}\subset \mathbb{R}_+$ with $\delta_j\to0^+$, $\{z^{\delta_j}\}_{j\geq1}
  \subset L^2(0,T;L^2(\Omega))$ be a sequence satisfying $\|z^{\delta_j}-z^\dag\|_{L^2(0,T;L^2(\Omega))}
  =\delta_j$ for some exact data $z^\dag=u(q^\dag)$, and $q_j^*$ be a minimizer to $J_{\gamma_j}
  (\cdot;z^{\delta_j})$ over $\mathcal{A}$. If $\{\gamma_j\}_{j\geq1}$ satisfies
  $\lim_{j\to\infty}\gamma_j =\lim_{j\to \infty}\frac{\delta_j^2}{\gamma_j} =0$,
      then $\{q_j^*\}_{j\geq 1}$ contains a subsequence converging to a minimum-$H^1(\Omega)$ seminorm solution
       in $H^1(\Omega)$.
\end{itemize}
\end{theorem}

\section{Numerical approximation and convergence analysis}\label{sec:fully}

Now we describe the discretization of problem \eqref{eqn:ob}--\eqref{eqn:var}
and show the convergence of the approximations.

\subsection{Numerical approximation}
First, we describe a spatially semidiscrete scheme for problem \eqref{eqn:fde} based on the
Galerkin FEM; see \cite{JinLazarovZhou:2019} for a recent overview on the numerical approximation
of the subdiffusion model. Let $\mathcal{T}_h$ be a shape regular quasi-uniform triangulation of the
domain $\Omega $ into $d$-simplexes, denoted by $T$, with a mesh size $h\in (0,1)$. Over $\mathcal{T}_h$,
we define a continuous piecewise linear finite element space $X_h$ by
\begin{equation*}
  X_h= \left\{v_h\in  {\dH1}:\ v_h|_T \mbox{ is a linear function}\ \forall\, T \in \mathcal{T}_h\right\},
\end{equation*}
and similarly the space $V_h$ by
\begin{equation*}
  V_h= \left\{v_h\in H^1(\Omega):\ v_h|_T \mbox{ is a linear function}\ \forall\, T \in \mathcal{T}_h\right\}.
\end{equation*}
The spaces $X_h$ and $V_h$ will be employed to approximate the state $u$ and the diffusion coefficient
$q$, respectively. We define the $L^2(\Omega)$ projection $P_h:L^2(\Omega)\to X_h$ by
\begin{equation*}
     (P_h \varphi,\chi) =(\varphi,\chi) , \quad \forall\, \chi\in X_h.
\end{equation*}
Note that the operator $P_h$ satisfies the following error estimate: for any $s\in[1,2]$,
\begin{equation*}
  \|P_h\varphi-\varphi\|_{L^2(\Omega)} + h\|\nabla(P_h\varphi-\varphi)\|_{L^2(\Omega)}\leq h^s\|\varphi\|_{\dot H^s(\Omega)},\quad \forall \varphi\in \dot H^s(\Omega).
\end{equation*}
Let $\mathcal{I}_h$ be the interpolation operator associated with the finite element space $V_h$. Then it has
the following error estimates for $s=1,2$ (see e.g., \cite[Theorem 1.103]{ern-guermond}):
\begin{align}
  \|v-\mathcal{I}_hv\|_{L^2(\Omega)} + h\|v-\mathcal{I}_hv\|_{H^1(\Omega)} & \leq ch^2 \|v\|_{H^2(\Omega)}, \quad \forall v\in H^2(\Omega),\label{eqn:int-err-2}\\
  \|v-\mathcal{I}_hv\|_{L^\infty(\Omega)} + h\|v-\mathcal{I}_hv\|_{W^{1,\infty}(\Omega)} & \leq {ch^s \|v\|_{W^{s,\infty}(\Omega)},} \quad \forall v\in W^{s,\infty}(\Omega).\label{eqn:int-err-inf}
\end{align}

Now we partition the time interval $[0,T]$ uniformly, with grid points $t_n=n\tau$, $n=0,\ldots,N$, and
a time step size $\tau=T/N$. The fully discrete scheme for problem \eqref{eqn:fde} reads: Given $U_h^0=P_hu_0\in X_h$, find $U_h^n\in X_h$ such that
\begin{align}\label{eqn:fully-0}
  (\bar \partial_\tau^\alpha (U_h^n-U_h^0),\chi)+(q \nabla U_h^n, \nabla \chi)= (f^n,\chi),\quad \forall\chi\in X_h, \quad \quad n=1,2,\ldots,N,
\end{align}
where $f^n=\frac1\tau \int_{t_{n-1}}^{t_n}  f(s)\,\d s$ and $\bar\partial_\tau^\alpha \varphi^n$ denotes the
backward Euler convolution quadrature (CQ) approximation (with $\varphi^j=\varphi(t_j)$):
\begin{equation}\label{eqn:CQ-BE}
  \bar\partial_\tau^\alpha \varphi^n = \tau^{-\alpha} \sum_{j=0}^nb_j^{(\alpha)}\varphi^{n-j},\quad\mbox{ with } (1-\xi)^\alpha=\sum_{j=0}^\infty b_j^{(\alpha)}\xi^j.
\end{equation}
{Note that the weights $b_j^{(\alpha)}$ are given explicitly by
$b_j^{(\alpha)} = (-1)^j\frac{\Gamma(\alpha+1)}{\Gamma(\alpha-j+1)\Gamma(j+1)}$, and thus
\begin{equation*}
 b_j^{(\alpha)} = (-1)^j{\frac{\alpha(\alpha-1)\cdots(\alpha-j+1)}{j!}},\quad j\geq 1,
\end{equation*}
from which it can be verified directly that $b_0^{(\alpha)}=1$ and $b_j^{(\alpha)}<0$ for $j\geq 1$.
Similarly, one deduces $b_j^{(\alpha-1)}>0$, for $j=0,1,\ldots$. Meanwhile, by definition, we have (with $\varphi^0=0$)
\begin{equation*}
\bar\partial_\tau^{\alpha-1}\bar\partial_\tau \varphi^n = \tau^{-\alpha}\sum_{j=1}^n b_{n-j}^{(\alpha-1)} (\varphi^{j} - \varphi^{j-1})
= \tau^{-\alpha}\Big(  b_{0}^{(\alpha-1)} \varphi^{n} + \sum_{j=1}^{n-1} (b_{n-j}^{(\alpha-1)}- b_{n-j-1}^{(\alpha-1)}) \varphi^{j} \Big).
\end{equation*}
Direct computation shows $b_{j}^{(\alpha-1)} - b_{j-1}^{(\alpha-1)} = b_{j}^{(\alpha)}$. This and the fact
$b_{0}^{(\alpha-1)} = b_{0}^{(\alpha)} = 1$ shows the following associativity of convolution quadrature (with $\varphi^0=0$)
\begin{equation}\label{eqn:cq-ass}
 \bar\partial_\tau^{\alpha-1}\bar\partial_\tau \varphi^n  =   \bar\partial_\tau^{\alpha}  \varphi^{n}.
\end{equation}
}
Upon letting the discrete operator $A_h(q):X_h\to X_h$ by $-(A_h(q)v_h,\chi)=(q\nabla v_h,\nabla \chi)$
for all $v_h,\chi\in X_h$, the fully discrete scheme \eqref{eqn:fully-0} can be rewritten as
\begin{equation*}
  \bar \partial_\tau^\alpha (U_h^n-U_h^0)-{A_h(q) U_h^n} = P_hf^n, \quad n=1,2,\ldots,N.
\end{equation*}

Now we can formulate the finite element discretization of problem \eqref{eqn:ob}--\eqref{eqn:var}:
\begin{equation}\label{eqn:ob-disc}
    \min_{q_h \in \Uad_h} J_{\gamma,h,\tau}(q_h)=\frac\tau2 \sum_{n=1}^N \int_\Omega |U_h^n(q_h) - z_n^\delta|^2 \,\d x   + \frac\gamma2\|\nabla q_h \|_{L^2(\Omega) }^2,
\end{equation}
with $z_n^\delta=\tau^{-1}\int_{t_{n-1}}^{t_n}z^\delta \d t$, and $U_h^n(q_h)$ satisfying $U_h^0(q_h)=P_hu_0$ and
\begin{align}\label{eqn:fully}
  \bar \partial_\tau^\alpha (U_h^n(q_h)-U_h^0) - {A_h(q_h) U_h^n(q_h)} = P_hf^n, \quad n=1,2,\ldots,N.
\end{align}
The discrete admissible set $\Uad_h$ is taken to be
\begin{equation*}
    \Uad_h=\{ q_h\in V_h:{c_0}\le q_h(x) \le {c_1}\quad \text{in}~~\Omega  \}.
\end{equation*}
Clearly, $\Uad_h=\Uad\cap V_h$. Problem \eqref{eqn:ob-disc}--\eqref{eqn:fully} is a finite-dimensional
nonlinear optimization problem with PDE and box constraint, and can be solved efficiently.
The analysis of problem \eqref{eqn:ob-disc}--\eqref{eqn:fully}
is the main focus of Sections \ref{ssec:conv} and \ref{sec:err}.

\subsection{Existence and convergence}\label{ssec:conv}
This part is devoted to the convergence analysis of the discrete approximations given by
the scheme \eqref{eqn:ob-disc}--\eqref{eqn:fully} to the continuous formulation \eqref{eqn:ob}--\eqref{eqn:var}.
We begin with some \textit{a priori} estimate on the solutions of the time-stepping scheme \eqref{eqn:fully-0}.
The proof relies on positivity of CQ.
\begin{lemma}\label{lem:disc-pos}
Let $V_{h}^n\in X_h$ solve
\begin{align*}
    (\bar \partial_\tau^\alpha V_h^n,\chi)+(q_h \nabla V_h^n, \nabla \chi) & = (f_h^n,\chi),\quad \forall \chi\in X_h, \quad n=1,2,\ldots,N,
\end{align*}
with $V_h^0 = 0$. Then there holds
\begin{equation*}
   \tau\sum_{n=1}^N (\nabla V_h^n, \nabla V_h^n) \leq c\tau \sum_{n=1}^N (f_h^n,V_h^n) .
\end{equation*}
\end{lemma}
\begin{proof}
Upon letting $\chi=V_h^n\in X_h$ and then summing over $n$ leads to
\begin{equation*}
  \tau \sum_{n=1}^N(\bar\partial_\tau^\alpha V_h^n,V_h^n) + \tau\sum_{n=1}^N (q_h\nabla V_h^n,{\nabla}V_h^n) = \tau \sum_{n=1}^N (f_h^n,V_h^n).
\end{equation*}
Now we shall show that the first term on the left hand side is nonnegative. To this end, we extend
$\{ V_h^n \}_{n=0}^N$ to $\{ V_h^n \}_{n=-\infty}^{n=\infty}$ and $\{ b_n^{(\alpha)} \}_{n=0}^{n=\infty}$ to
$\{ b_n^{(\alpha)}  \}_{n=-\infty}^{n=\infty}$ by zero. Then $\bar\partial_\tau^\alpha V_h^n$ can be written as
$
\bar \partial_\tau^{\alpha} V_h^n = \tau^{-\alpha}\sum_{j=-\infty}^\infty b_{n-j}^{(\alpha)} V_h^j.
$
Next we denote the discrete Fourier transform $ \widetilde {[V_h^n]}(\zeta)$ by
$\widetilde {[V_h^n]}(\zeta) = \sum_{n=-\infty}^\infty V_h^n  e^{-\mathrm{i} n\zeta}$.
By Parseval's theorem, since $V_h^0=0$, we have
\begin{equation*}
 \sum_{j=1}^N (\bar \partial_\tau^{\alpha} V_h^n , V_h^n) = \frac{1}{2\pi}\int_{-\pi}^\pi  (\widetilde{[\bar \partial_\tau^{\alpha} V_h^n]}(\zeta) , \widetilde{[V_h^n]}(\zeta)^*) \,\d\zeta
\end{equation*}
By the property of discrete Fourier transform, we have
\begin{equation*}
\begin{split}
 \sum_{j=1}^N (\bar \partial_\tau^{\alpha} V_h^n , V_h^n) &= \frac{\tau^{-\alpha}}{2\pi}\int_{-\pi}^\pi \widetilde{[b_n^{(\alpha)}]}~ \Big|\widetilde{[V_h^n]}(\zeta)\Big|^2 \,\d\zeta
 = \frac{\tau^{-\alpha}}{\pi}\int_{0}^\pi \Big[{\Re}\Big( 1-e^{-\mathrm{i}\zeta} \Big)^{\alpha}\Big] \Big|\widetilde{[y_n]}(\zeta)\Big|^2 \,\d\zeta \ge 0.\\
 \end{split}
\end{equation*}
Then Cauchy-Schwarz inequality and Poincar\'{e}'s inequality imply the desired estimate.
\end{proof}

\begin{lemma}\label{lem:CQ-0}
The following statements hold
\begin{align*}
  \sum_{n=0}^mb_n^{(\alpha)} = b_m^{(\alpha-1)} \quad\mbox{and}\quad
  |\tau^{-\alpha} \sum_{n=0}^{m} b_{n}^{(\alpha)}| \leq ct_{m+1}^{-\alpha}.
\end{align*}
\end{lemma}
\begin{proof}
Let $\sum_{n=0}^mb_n^{(\alpha)} = v_m$. Then by changing the order of summation, we have
\begin{align*}
 \sum_{m=0}^\infty v_m \xi^m &= \sum_{m=0}^\infty \xi^m \sum_{n=0}^mb_n^{(\alpha)} {=\sum_{n=0}^\infty b_n^{(\alpha)}  \sum_{m=n}^\infty \xi^m}\\
   &= {\sum_{n=0}^\infty b_n^{(\alpha)} \xi^n \sum_{m=n}^\infty \xi^{m-n}} = \Big(\sum_{n=0}^\infty b_n^{(\alpha)}\xi^n\Big) \Big(\sum_{m=0}^\infty \xi^m\Big)\\
 &= (1-\xi)^{\alpha} (1-\xi)^{-1} = (1-\xi)^{\alpha-1}.
\end{align*}
Therefore,  $v_m = b_m^{(\alpha-1)} \le c(m+1)^{-\alpha}$ \cite[Lemma 2.3]{JinLiZhou:nonlinear}, which shows the
second assertion.
\end{proof}

The next result gives a discrete analogue of the following well known inequality \cite{Alikanov:2010}
\begin{equation*}
   \fy(t) \partial_t^\alpha(\fy(t)-\fy(0))\geq \tfrac12\partial_t^\alpha(|\fy(t)|^2-\fy(0)|^2).
\end{equation*}
It is useful for deriving \textit{a priori} estimates on the fully discrete solutions.
\begin{lemma}\label{lem:cq-prod}
Let $\bar\partial_\tau^\alpha \fy^n$ be the backward Euler CQ defined as \eqref{eqn:CQ-BE}. Then there holds
\begin{equation*}
\big(\bar\partial_\tau^\alpha (\fy^n - \fy^0)\big)\fy^n \ge \tfrac12 \bar\partial_\tau^\alpha \big(  | \fy^n |^2- | \fy^0 |^2 \big)
\end{equation*}
\end{lemma}
\begin{proof}
By the definition of backward Euler CQ in \eqref{eqn:CQ-BE}, we deduce
\begin{equation*}
\begin{split}
\big(\bar\partial_\tau^\alpha (\fy^n - \fy^0)\big)\fy^n
&=\tau^{-\alpha}\Big( |\fy^n|^2 + \sum_{j=0}^{n-1} b_{n-j}^{(\alpha)} \fy^n \fy^j  - \big(\sum_{j=0}^{n} b_{n-j}^{(\alpha)}\big) \fy^n\fy^0\Big).
 \end{split}
\end{equation*}
Now since the binomial coefficient $b_{j}^{(\alpha)}<0$ for $j\geq1$, we deduce
\begin{equation*}
\begin{split}
 \sum_{j=0}^{n-1} b_{n-j}^{(\alpha)} \fy^n \fy^j  \ge \frac12 \sum_{j=0}^{n-1} b_{n-j}^{(\alpha)} |\fy^n|^2  +  \frac12 \sum_{j=0}^{n-1} b_{n-j}^{(\alpha)} |\fy^j|^2,
 \end{split}
\end{equation*}
and
\begin{equation*}
 \big(\sum_{j=0}^{n} b_{n-j}^{(\alpha)}\big) \fy^n\fy^0 \le \frac12 \big(\sum_{j=0}^{n} b_{n-j}^{(\alpha)}\big) |\fy^n|^2 + \frac12 \big(\sum_{j=0}^{n} b_{n-j}^{(\alpha)}\big) |\fy^0|^2.
 \end{equation*}
Then the desired result follows immediately.
\end{proof}

The next result gives a discrete continuity result.
\begin{lemma}\label{lem:cont-p2s-disc}
Let the sequence $\{q_h^j\}\subset \Uad_h$ be convergent to $q_h^*\in \Uad_h$ in $L^1(\Omega)$. Then
\begin{equation*}
  \lim_{j\to\infty}\tau\sum_{n=1}^N\int_\Omega |U_{h,\tau}^n(q_h^j) -z_n^\delta|^2\d x = \tau\sum_{n=1}^N\int_\Omega|U_{h,\tau}^n(q_h^*) -z_n^\delta|^2\d x
\end{equation*}
\end{lemma}
\begin{proof}
Using  \cref{lem:disc-pos}, the proof is similar to that of \cref{lem:conti-p2s}, since in a finite-dimensional
space $V_h$, all norms are equivalent, and the convergence in $L^1(\Omega)$ implies almost every convergence
\cite{EvansGariepy:2015}. Thus the proof is omitted.
\end{proof}

Then we can obtain the existence of a discrete minimizer $q_h^*\in \Uad_h$. The proof is identical with that in
\cref{thm:ex}, and hence omitted. Note that the discrete minimizer $q_h^*$ depends implicitly also on the
time step size $\tau$ through the weak formulation \eqref{eqn:fully}.
\begin{theorem}\label{thm:ex-fully}
Under \cref{ass:data1}, there exists one minimizer $q_h^*\in \Uad_h$ to \eqref{eqn:ob-disc}--\eqref{eqn:fully}.
\end{theorem}

Below we analyze the convergence of the sequence $\{q_h^*\}_{h>0}$ as $h,\tau\to0$.
The next result is an analogue of \cref{lem:conti-p2s}, and
plays an important role in the convergence analysis. For the sequence of discrete solutions
$U_{h,\tau}^n\equiv U_{h,\tau}^n(q_h)\in X_h$ to problem \eqref{eqn:fully}, we define a piecewise constant in time
interpolation $u_{h,\tau}(t)$ by
\begin{equation}
  u_{h,\tau}(t) = U_{h,\tau}^n, \quad t\in [t_{n},t_{n+1}),\quad n=0,\ldots, N-1.\label{eqn:uhtau}
\end{equation}

\begin{lemma}\label{lem:cont-p2s-cont}
Let $U_{h,\tau}^n\equiv U_{h,\tau}^n(q_h)\in X_h$ be the discrete solutions to problem \eqref{eqn:fully} with $q_h\in \Uad_h$,
and the sequence $\{q_h\in\Uad_h\}_{h>0}$ convergent to some $q^*\in \Uad$ a.e.
{as $h,\tau \to 0^+$}. Then under  \cref{ass:data1}, for the piecewise constant
interpolation $u_{h,\tau}$ defined in \eqref{eqn:uhtau}, there holds
\begin{equation*}
   u_{h,\tau}(q_h)\to u(q^*) \mbox{ weakly in } L^2(0,T;H^1(\Omega)),\quad \mbox{as } h,\tau\to0.
\end{equation*}
\end{lemma}
\begin{proof}
Taking the test function $\chi=U_{h}^n-U_{h}^0$ in \eqref{eqn:fully} and summing over $n$ yield
\begin{equation*}
  \tau\sum_{n=0}^N (\bar\partial_\tau^\alpha (U_h^n-U_h^0),U_h^n-U_h^0) + \tau\sum_{n=1}^N (q_h\nabla U_h^n,\nabla(U_h^n-U_h^0)) = \tau \sum_{n=1}^n(f_h^n,U_h^n-U_h^0),
\end{equation*}
This identity, the nonnegativity of the discrete convolution $\bar\partial_\tau^\alpha$ (see the proof of
\cref{lem:disc-pos}),  Poincar\'{e} inequality and Young's inequality, and the $L^2(\Omega)$ stability
of $P_h$ lead to
\begin{align*}
  \tau\sum_{n=1}^N \|\nabla U_h^n\|_{L^2(\Omega)}^2 & \leq c\tau\sum_{n=1}^N \Big(\|\nabla U_h^0\|_{L^2(\Omega)}^2 + \|f_h^n\|_{H^{-1}(\Omega)}^2\Big) \leq c\big(\|\nabla u_0\|_{L^2(\Omega)}^2 + \|f\|_{L^2(0,T;H^{-1}(\Omega))}^2\big).
\end{align*}
Thus, the sequence $\{u_{h,\tau}\}_{h,\tau>0}$ is uniformly bounded in $L^2(0,T;H^1(\Omega))$, and thus
there exists a subsequence, still denoted by $\{u_{h,\tau}\}_{h,\tau>0}$, such that
\begin{equation}\label{eqn:weak-conv-u}
  u_{h,\tau} \mbox{ converges weakly to some } u^* \mbox{ in } L^2(0,T;H^1(\Omega)).
\end{equation}
Meanwhile, by taking the test function $\chi=\bar \partial_\tau^\alpha (U_h^n-U_h^0)$ in \eqref{eqn:fully},
\begin{align}\label{eqn:es0}
  \tau\sum_{n=0}^N (\bar\partial_\tau^\alpha(U_h^n-U_h^0),\bar\partial_\tau^\alpha(U_h^n-U_h^0)) + \tau\sum_{n=1}^N (q_h\nabla U_h^n,\bar\partial_\tau^\alpha\nabla(U_h^n-U_h^0))
  = \tau\sum_{n=1}^n(f_h^n,\bar\partial_\tau^\alpha(U_h^n-U_h^0)).\nonumber
\end{align}
Then \cref{lem:cq-prod}, the fact that $b_j^{(\alpha-1)}>0$ for all $j\geq 0$ and \cref{lem:CQ-0} lead to
\begin{align*}
&2\tau\sum_{n=1}^N(\nabla U_h^n)^\mathrm{t}\bar\partial_\tau^\alpha\nabla(U_h^n-U_h^0)
\geq \tau\sum_{n=1}^N  \bar\partial_\tau^\alpha \Big( \| \nabla U_h^n \|_{L^2(\Omega)}^2  - \| \nabla U_h^0\|_{L^2(\Omega)}^2 \Big)\\
=& \tau\sum_{j=0}^N  \Big(\|\nabla U_h^j\|_{L^2(\Omega)}^2  - \| \nabla U_h^0\|_{L^2(\Omega)}^2 \Big) b_{N-j}^{(\alpha-1)}
\geq\tau\sum_{j=0}^N  - \|\nabla U_h^0\|_{L^2(\Omega)}^2  b_{N-j}^{(\alpha-1)}
\ge - c \| \nabla U_h^0 \|_{L^2(\Omega)}^2,
\end{align*}
Hence, there holds
$
\tau\sum_{n=1}^N (q_h\nabla U_h^n,\bar\partial_\tau^\alpha\nabla(U_h^n-U_h^0))
\ge - c \| \nabla U_h^0 \|_{L^2\II}^2.
$
This and Young's inequality imply
\begin{equation*}
  \tau \sum_{n=1}^N \|\bar\partial_\tau^\alpha( U_{h}^n-U_h^0)\|_{L^2(\Omega)}^2 \leq c(\|\nabla u_0\|_{L^2(\Omega)}^2 + \|f\|_{L^2(0,T;L^2(\Omega))}^2).
\end{equation*}
Thus the sequence of piecewise constant interpolation, denoted by $\{\bar\partial_\tau^\alpha (u_{h,\tau}-U_h^0)\}_{h,\tau>0}$,
is uniformly bounded in $L^2(0,T;L^2(\Omega))$, and there exists a subsequence, still denoted by $\{\bar\partial_\tau^\alpha
(u_{h,\tau}-U_h^0)\}_{h,\tau>0}$, and some $v^*\in L^2(0,T;L^2(\Omega))$ such that it converges to $v^*$ weakly in $L^2(0,T;L^2
(\Omega))$. Next we claim that $u^*$ satisfies the weak formulation of $u(q^*)$, cf. \eqref{eqn:var}. To this end, we take a
smooth test function $\phi\in C^1([0,T];{\dH1})$ with $\phi(T)=0$, and define an approximation $\phi_{h,\tau}$ by
$\phi_{h,\tau} (t) = \tau^{-1}\int_{t_{n-1}}^{t_n}P_h\phi(t)\d t$, $t\in (t_{n-1},t_n]$. Then the density of $X_h$ in $\dot{H}^1(\Omega)$
and piecewise constant functions in $L^2(0,T)$ implies that $\lim_{h,\tau\to0^+}   \|\phi_{h,\tau}-\phi\|_{L^2(0,T;H^1(\Omega))}=0$.
Hence, by discrete summation by parts and straightforward computation, there holds
\begin{align*}
  \tau\sum_{n=1}^N (\bar\partial_\tau^\alpha(U_h^n-U_h^0), \phi_{h,\tau}(t_{n})) & = (\bar\partial_\tau^\alpha(u_{h,\tau}-U_h^0),P_h\phi(t))_{L^2(0,T;L^2(\Omega))} \\
   &= (u_{h,\tau}-U_h^0,{^R\bar\partial_\tau^\alpha}P_h\phi(t))_{L^2(0,T;L^2(\Omega))},
\end{align*}
where the notation ${^R\bar\partial_\tau^\alpha}P_h\phi(t)$ denotes
${^R\bar\partial_\tau^\alpha}P_h\phi(t) = {\sum_{i=n}^N b_{n-i}^{(\alpha)} P_h\phi(t+(i-n)\tau)}$,
for $ t\in(t_{n-1},t_n], \ n=1,2,\ldots,N.$
By the approximation property of $^R\bar\partial_\tau^\alpha$ and $P_h$ {(see, e.g., \cite[Section 2.2]{Podlubny:1999})}, since $\phi\in C^1([0,T];{\dH1})$,
${^R\bar\partial_\tau^\alpha}P_h\phi(t)$ converges to $^R\partial_t^\alpha\phi(t)$ in $L^2(0,T;L^2(\Omega))$ as $h,\tau\to0^+$, and
\begin{equation*}
  \lim_{h,\tau\to0}(u_{h,\tau}-U_h^0,{^R\bar\partial_\tau^\alpha}P_h\phi(t))_{L^2(0,T;L^2(\Omega))} = (u^*-u_0, {^R\partial_t^\alpha}\phi(t))_{L^2(0,T;L^2(\Omega))},
\end{equation*}
and meanwhile, by the weak convergence of $\bar\partial_\tau^\alpha(u_{h,\tau}-U_h^0)$ to $v^*$ in $L^2(0,T;L^2(\Omega))$ and
the approximation property of $P_h$,
\begin{equation*}
  \lim_{h,\tau\to0} (\bar\partial_\tau^\alpha(u_{h,\tau}-U_h^0),P_h\phi(t))_{L^2(0,T;L^2(\Omega))} = (v^*,\phi(t))_{L^2(0,T;L^2(\Omega))}.
\end{equation*}
Comparing the preceding two identities shows that  $v^*=\partial_t^\alpha(u^*-u_0)$, i.e., $v^*$ is the weak fractional order derivative
of $u^*-u_0$. Now taking the test function $\chi=\phi_{h,\tau}(t_n)$ in \eqref{eqn:fully} and summing over $n$, we obtain
\begin{equation*}
  \tau\sum_{n=0}^N (\bar\partial_\tau^\alpha (U_h^n-U_h^0),\phi_{h,\tau}(t_n)) + \tau\sum_{n=1}^N (q_h\nabla U_h^n,\nabla\phi_{h,\tau}(t_n)) = \tau \sum_{n=1}^N(f_h^n,\phi_{h,\tau}(t_n)),
\end{equation*}
and by the definition of piecewise constant interpolations $\bar\partial_\tau(U_{h,\tau}^n-U_h^0)$
and $u_{h,\tau}(t)$ and the construction of the test function $\phi_{h,\tau}(t_n)$, it is equivalent to
\begin{equation*}
  (\bar\partial_\tau^\alpha (u_{h,\tau}^n-U_h^0),P_h\phi)_{L^2(0,T;L^2(\Omega))} + (q_h\nabla u_{h,\tau},\nabla P_h\phi(t))_{L^2(0,T;L^2(\Omega))} =
  (f_{h,\tau}, P_h\phi(t))_{L^2(0,T;L^2(\Omega))},
\end{equation*}
where $f_{h,\tau}(t)= \tau^{-1}\int_{t_{n-1}}^{t_n}P_hf(t)\d t$, for $t\in (t_{n-1},t_n]$, $n=1,\ldots,N$,
for which there holds $\lim_{h,\tau\to0^+}  \| f_{h,\tau} - f\|_{L^2(0,T;L^2(\Omega))} = 0$.
Upon passing limit on both sides, we deduce
\begin{align*}
  \lim_{h,\tau\to0^+}(\bar\partial_\tau^\alpha (U_{h,\tau}^n-U_h^0),P_h\phi)_{L^2(0,T;L^2(\Omega))} &= (\partial_t^\alpha(u^*-u_0),\phi)_{L^2(0,T;L^2(\Omega))},\\
  \lim_{h,\tau\to0^+} (f_{h,\tau}, P_h\phi(t))_{L^2(0,T;L^2(\Omega))} &= (f,\phi)_{L^2(0,T;L^2(\Omega))}.
\end{align*}
Further, to analyze the term $(q_h\nabla u_{h,\tau},\nabla P_h\phi(t))_{L^2(0,T;L^2(\Omega))}$, we employ the following splitting
\begin{align*}
  & |(q_h\nabla u_{h,\tau},\nabla P_h\phi(t))_{L^2(0,T;L^2(\Omega))} - ( q^*\nabla u^*,\nabla \phi(t))_{L^2(0,T;L^2(\Omega))}|\\
  \leq & |(q_h\nabla u_{h,\tau},\nabla P_h\phi(t))_{L^2(0,T;L^2(\Omega))} - ( q_h\nabla u_{h,\tau},\nabla \phi(t))_{L^2(0,T;L^2(\Omega))}| \\
      & + |(q_h\nabla u_{h,\tau},\nabla \phi(t))_{L^2(0,T;L^2(\Omega))} - ( q^* \nabla u_{h,\tau},\nabla \phi(t))_{L^2(0,T;L^2(\Omega))}| \\
      & + |(q^*\nabla u_{h,\tau},\nabla \phi(t))_{L^2(0,T;L^2(\Omega))} - ( q^* \nabla u^*,\nabla \phi(t))_{L^2(0,T;L^2(\Omega))}| : = {\rm I} + {\rm II} + {\rm III}.
\end{align*}
We bound the three terms separately. By the approximation property of $P_h$ and uniform boundedness of
$u_{h,\tau}$ in $L^2(0,T;H^1(\Omega))$ due to \eqref{eqn:weak-conv-u}, we deduce
\begin{equation*}
  \lim_{h,\tau\to0^+} {\rm I} \leq \lim_{h,\tau\to0^+}c\|u_{h,\tau}\|_{L^2(0,T;H^1(\Omega))}\|P_h\phi-\phi\|_{L^2(0,T;H^1(\Omega))}=0.
\end{equation*}
Next, since $q_h$ converges to $q^*$ a.e. and \eqref{eqn:weak-conv-u}, by
dominated convergence theorem \cite[Theorem 1.9]{EvansGariepy:2015} (with the argument in
Lemma \ref{lem:conti-p2s}), we have
\begin{equation*}
   \lim_{h,\tau\to0^+} {\rm II} \leq \lim_{h,\tau\to0^+}\|u_{h,\tau}\|_{L^2(0,T;H^1(\Omega))}\|(q_h-q^*)\phi\|_{L^2(0,T;H^1(\Omega))} = 0.
\end{equation*}
The third term ${\rm III}$ tends to zero as $h,\tau\to0^+$, in view of the weak convergence in \eqref{eqn:weak-conv-u}.
Consequently, combining the three assertions together yields
\begin{equation*}
  \lim_{h,\tau\to0^+}(q_h\nabla u_{h,\tau},\nabla P_h\phi(t))_{L^2(0,T;L^2(\Omega))}= ( q^*\nabla u^*,\nabla \phi(t))_{L^2(0,T;L^2(\Omega))}.
\end{equation*}
In sum, the limit $u^*$ satisfies that for any $\phi\in C^1([0,T];{\dot H^1(\Omega)}) $, there holds
\begin{equation*}
  (\partial_t^\alpha (u^*-u_0), \phi)_{L^2(0,T;L^2(\Omega))} + (q^*\nabla u^*,\nabla \phi)_{L^2(0,T;L^2(\Omega))} = (f,\phi)_{L^2(0,T;L^2(\Omega))}.
\end{equation*}
By the density of the space $C^1([0,T]; {\dH1})$ in $L^2(0,T; {\dH1})$,
the identity holds also for any $\phi\in L^2(0,T;{\dH1})$.
This immediately shows that $u^*$ is a weak solution to problem \eqref{eqn:fde} with $q^*$, i.e., $u^*=u(q^*)$.
Since every subsequence contains a convergent sub-subsequence, the whole sequence converges to $u(q^*)$.
This completes the proof of the lemma.
\end{proof}

Now we can state the main result of this part, i.e., the convergence of the
discrete solutions $\{q_h^*\}_{h>0}$ to the continuous optimization problem \eqref{eqn:ob}--\eqref{eqn:var}.
{With Lemma \ref{lem:cont-p2s-cont} at hand, the proof is standard and it is included only for completeness.}
\begin{theorem}\label{thm:conv-fully}
Let $\{q_h^*\}_{h>0}$ be a sequence of minimizers to problem \eqref{eqn:ob-disc}--\eqref{eqn:fully}.
Then under  \cref{ass:data1}, it contains a subsequence convergent to a minimizer of problem
\eqref{eqn:ob}--\eqref{eqn:var} in $H^1(\Omega)$.
\end{theorem}
\begin{proof}
Since the constant function $q_h\equiv c_0$ belongs to the admissible set $\mathcal{A}_h$ for any $h$, there holds $J_{\gamma,h,\tau}
(q_h^*)\leq J_{\gamma,h,\tau}(c_0)<\infty$, from which it directly follows that the sequence $\{q_h^*\}_{h>0}$ is uniformly bounded
in the  $H^1(\Omega)$-seminorm. This and the box constraint in $\mathcal{A}_h$ imply that the sequence $\{q_h^*\in \Uad_h\}_{h>0}$
is uniformly bounded in the $H^1(\Omega) $ norm. Thus there exists a subsequence, still denoted by $\{q_h^*
\}_{h>0}$ such that it converges weakly in the $H^1(\Omega)$ to some $q^*\in \mathcal{A}$. We claim that $q^*$ is a
minimizer to problem \eqref{eqn:ob}--\eqref{eqn:var}. For any $q\in \Uad$, by the density of $W^{1,\infty}(
{\Omega})$ in $H^1(\Omega)$ \cite{EvansGariepy:2015} (e.g., by means of mollifier), there exists a sequence
$\{q^\epsilon\}_{\epsilon>0}\subset \Uad\cap W^{1,\infty}({\Omega})$ such that $\lim_{\epsilon\to0^+}
\|q^\epsilon-q\|_{H^1(\Omega)}=0$ and almost everywhere. Now let $q_h^\epsilon = \mathcal{I}_h q^\epsilon\in V_h$.
By the minimizing property of $q_h^*$, there holds
\begin{equation}\label{eqn:min-h}
  J_{\gamma,h,\tau}(q^*_h) \leq J_{\gamma,h,\tau}(q_h^\epsilon).
\end{equation}
By the weak lower semi-continuity of norms, we have $\|\nabla q^*\|_{L^2(\Omega)}  \leq \liminf_{h\to0}\|\nabla q_h^*\|_{L^2(\Omega)}$.
Similarly, by the weak convergence of $u_{h,\tau}(q_h^*)$ to $u(q^*)$ in $L^2(0,T;H^1(\Omega))$
in  \cref{lem:cont-p2s-cont} and the embedding $H^1(\Omega)\hookrightarrow L^2(\Omega)$, and the construction
of the function $z_\tau^\delta(t)=
\tau^{-1}\int_{t_{n-1}}^{t_n}z^\delta(t)\d t$, for $t\in(t_{n-1},t_n]$, $n=1,\ldots,N$,
$\lim_{\tau\to0^+}\|z^\delta - z_\tau^\delta\|_{L^2(0,T;L^2(\Omega))}=0$, we have
$\|u(q^*)-z^\delta\|^2_{L^2(0,T;L^2(\Omega))} \leq \liminf_{h,\tau\to0^+}\|u_{h,\tau}(q_h^*)-z_\tau^\delta\|_{L^2(0,T;L^2(\Omega))}
  = \liminf_{h,\tau\to0^+}\tau\sum_{n=1}^N\|U_h^n(q_h^*)-z_n^\delta\|_{L^2(\Omega)}^2$,
and thus
\begin{equation}\label{eqn:conv-1}
  J_\gamma(q^*) \leq \lim_{h,\tau\to 0^+} J_{\gamma,h,\tau}(q_h^*).
\end{equation}
Meanwhile, by {Lemma \ref{lem:cont-p2s-cont}} and the approximation property of the operator $\mathcal{I}_h$
in \eqref{eqn:int-err-inf},
\begin{equation}\label{eqn:conv-2}
  \lim_{h,\tau\to 0^+} J_{\gamma,h,\tau}(q_h^\epsilon) = J_\gamma(q^\epsilon).
\end{equation}
Thus, taking limit as $h,\tau\to0^+$ in the inequality \eqref{eqn:min-h} yields
$J_{\gamma}(q^*) \leq J_{\gamma}(q^\epsilon)$. Further, since $q^\epsilon\to q$
in $H^1(\Omega)$ and almost everywhere as $\epsilon \to 0^+$, by  \cref{lem:conti-p2s}, there holds
\begin{equation}\label{eqn:conv-3}
  \lim_{\epsilon\to0^+} J_\gamma(q^\epsilon) = J_\gamma(q).
\end{equation}
Combining the relations \eqref{eqn:conv-1}--\eqref{eqn:conv-3} yields $J_\gamma(q^*)
\leq J_\gamma(q)$ for any $q\in \Uad$. This shows the weak convergence to a minimizer $q^*$
in $H^1(\Omega)$. Meanwhile, by the weak lower semi-continuity of the norms and a standard
argument by contradiction \cite{ItoJin:2015}, we have
$\lim_{h,\tau\to0^+}\|\nabla q_h^*\|_{L^2(\Omega)}^2 = \|\nabla q^*\|_{L^2(\Omega)}^2$.
Hence, the subsequence $\{q_h^*\}_{h>0}$ converges to $q^*$ in $H^1(\Omega)$. This completes
the proof of the theorem.
\end{proof}

\begin{remark}\label{rmk:alternative}{
Note that the continuity results in \cref{lem:conti-p2s}
and \cref{lem:cont-p2s-cont} are stated with respect to almost everywhere convergence {\rm(}deduced from
the $L^1(\Omega)$ convergence of the sequence of the diffusion coefficient{\rm)}, which can be induced by other penalties with
the underlying space compactly embedding into the space $L^1(\Omega)$, including the space of bounded variation
\cite{EvansGariepy:2015}. Thus, upon minor modifications, the results in Sections \ref{sec:cont} and \ref{sec:fully} hold also for
related regularized formulations, e.g., total variation penalty, which is suitable for recovering
discontinuous coefficients; see, e.g., \cite{Gutman:1990,CasasKogutLeugering:2016} for relevant studies for in the
parabolic and elliptic cases. Also note that the terminal observation $u(T)$ may require stronger regularity condition on the source $f$ than
Assumption \ref{ass:data1} so as to ensure $u(q)\in C([0,T];L^2(\Omega))$, depending on the value of the fractional
order $\alpha$.}
\end{remark}

\begin{remark}{Due to the nonlinearity of the parameter-to-state map $q\mapsto u(q)$, the regularized output
least-squares problem \eqref{eqn:ob}--\eqref{eqn:var} is expected to be highly nonconvex. Hence, numerically one can generally only
guarantee to reach a stationary point $\hat q_h$ of the optimality system (OS) when solving the discrete optimization problem
\eqref{eqn:ob-disc}--\eqref{eqn:fully}. One important theoretical question is the convergence of the sequence $\{\hat q_h\}_{h>0}$ of discrete
stationary points for OS. Note that the convergence analysis in \cref{sec:fully} relies essentially on extracting a
convergent subsequence of discrete minimizers $\{q_h^*\}_{h>0}$ in $L^1(\Omega)$, which in turn follows from the uniform
\textit{a priori} bound on $\{q_h^*\}_{h>0}$ in $H^1(\Omega)$, induced by the $H^1(\Omega)$-seminorm penalty. Thus, one
crucial step in extending the analysis to stationary points is to derive suitable uniform \textit{a priori} bound on
$\{\hat q_h\}_{h>0}$. This might be derived from the OS as follows. Indeed, the box constraints in the admissible set
$\mathcal{A}_h$ allows bounding the discrete state $U_h^n(\hat q_h)$ (and thus also the discrete adjoint) uniformly in
the discrete $L^2(0,T;H^1(\Omega))$ norm, and then the discrete variational inequality for $\hat{q}_h$
in OS allows uniformly bounding $\hat q_h$ in suitable Sobolev norm using ``elliptic'' regularity theory. We shall refrain from a detailed derivation of OS
and the associated convergence analysis for stationary points, since the analysis in \cref{sec:err} crucially exploits the minimizing property
of the discrete minimizer and does not extend to stationary points directly.}
\end{remark}

\section{Error estimates}\label{sec:err}
Now we derive error estimates of approximations $q_h^*$ under the following regularity
on the problem data.

\begin{assumption}\label{ass:data2}
The following conditions hold.
\begin{itemize}
\item[$\rm(i)$] $u_0\in {\dH2}$, $f\in C^2([0,T];L^2(\Omega))\cap L^\infty(0,T; \dH{\beta})$ with $\beta>\max(\frac{d}{2}-1,0)$,
and exact diffusion coefficient $q^\dag\in W^{2,\infty}(\Omega)$.
\item[$\rm(ii)$] $z^\delta\in C([0,T];L^2(\Omega))\cap C^2((0,T];L^2(\Omega))$
with $t^{1-\alpha}\|z^{\delta\prime}(t)\|_{L^2(\Omega)} + t^{2-\alpha}\|z^{\delta\prime\prime}(t)\|_{L^2(\Omega)} \le c$.
\end{itemize}
\end{assumption}
Under  \cref{ass:data2}(i), there exists a unique solution $u\in C([0,T];\dot H^2(\Omega))\cap C^2((0,T];L^2(\Omega))$
and for any $s\in[0,\beta)$ and $r\in[0,2]$, there holds
\begin{equation}\label{sol-reg}
  \| u(t) \|_{\dH{2}} + t^{\frac{s}{2}\alpha} \| u(t) \|_{\dH{2+s}} + t^{1-(1-\frac{s}{2})\alpha}\|u'(t)\|_{\dot H^s(\Omega)} + t^{2-\alpha}\|u''(t)\|_{L^2(\Omega)} \le c.
\end{equation}
See \cite{SakamotoYamamoto:2011,JinLazarovZhou:2019} for a proof of the regularity estimate.

The better temporal regularity on the observation $z^\delta$ and $u(q)$ enables slightly modifying the discrete
optimization problem $J_{h,\tau,\gamma}$, instead of using $z_n^\delta:=\tau^{-1}\int_{t_{n-1}}^{t_n}z^\delta(t)\d t$.
In particular, we can employ the trapezoid rule: with $a_0=a_N=1/2$ and $a_i=1$, $i=1,\ldots, N-1$,
\begin{equation}\label{eqn:ob-disc-1}
    \min_{q_h \in \Uad_h} J_{\gamma,h,\tau}(q_h)=\frac\tau2 \sum_{n=0}^N a_i\int_\Omega |U_h^n(q_h) - z^\delta(t_n)|^2 \,\d x   + \frac\gamma2\|\nabla q_h \|_{L^2(\Omega) }^2,
\end{equation}
subject to $q_h\in\Uad_h$ and $U_h^n(q_h)$ satisfying $U_h^0=P_hu_0$ and
\begin{align}\label{eqn:fully-1}
  \bar \partial_\tau^\alpha (U_h^n(q_h)-U_h^0) - A_h(q_h) U_h^n(q_h) = P_hf(t_n),\quad \quad n=1,2,\ldots,N.
\end{align}
This change allows deriving a better rate in $\tau$ in  \cref{thm:error-q} below. Under
\cref{ass:data2,thm:ex-fully,thm:conv-fully} in \cref{sec:fully} remain valid
for problem \eqref{eqn:ob-disc-1}--\eqref{eqn:fully-1}. The goal of this part is to derive error estimates for the approximation
constructed by \eqref{eqn:ob-disc-1}--\eqref{eqn:fully-1}.

We begin with some preliminary estimates under \cref{ass:data2}(i).

\begin{lemma}\label{lem:err-1}
Let $q^\dag$ be the exact diffusion coefficient and $u\equiv u(q^\dag)$ be the solution to problem \eqref{eqn:var},
and $\{U_h^n(q^\dag)\}$ and $\{U_h^n(\mathcal{I}_hq^\dag)\}$ be the solutions to the scheme
\eqref{eqn:fully-0} corresponding to $q^\dag$ and $\mathcal{I}_hq^\dag$, respectively. Then
under \cref{ass:data2}(i), with $\ell_h=|\log h|$,
\begin{equation*}
\begin{split}
\| u(t_n)-U_h^n(q^\dag)\|_{L^2(\Omega)} &\le c(\tau t_n^{\alpha-1} + h^2\ell_h),\\
\| u(t_n)-U_h^n(\mathcal{I}_h q^\dag)\|_{L^2(\Omega)} &\le c(\tau t_n^{\alpha-1} + h^2\ell_h).
\end{split}
\end{equation*}
\end{lemma}
\begin{proof}
The first estimate is immediate from \cite{JinLazarovZhou:SISC2016}
\begin{equation*}
\begin{split}
\|u(t_n)-U_h^n(q^\dag)\|_{L^2(\Omega)} &\le c h^2 \ell_h\Big(\| A(q^\dag) u_0 \|_{L^2(\Omega)} + \| f \|_{L^\infty(0,T;\dot H^\beta(\Omega))}\Big)\\
& + c \tau \Big( t_n^{\alpha-1} \| A(q^\dag) u_0+f(0) \|_{L^2(\Omega)} +  \int_0^{t_n}(t_n-s)^{\alpha-1}\| f'(s)\|_{L^2(\Omega)}\,\d s\Big).
\end{split}
\end{equation*}
To show the second estimate, we bound $\rho_h^n:=U_h^n(q^\dag)  - U_h^n(\mathcal{I}_h q^\dag)$, which satisfies $\rho_h^0 =0$ and
\begin{equation*}
 \bar \partial_\tau^\alpha \rho_h^n - A_h(q^\dag)  \rho_h^n =  [A_h(q^\dag)-A_h(\mathcal{I}_hq^\dag)]U_h^n(\mathcal{I}_hq^\dag),\quad n=1,2,\ldots,N,
\end{equation*}
where $A_h(q^\dag),A_h(\mathcal{I}_hq^\dag):X_h\to X_h$ are the discrete analogues of the elliptic operators $A(q^\dag)$
and $A(\mathcal{I}_hq^\dag) $ associated with $q^\dag$ and $\mathcal{I}_hq^\dag$, respectively. Thus, it can be written as
\begin{equation}\label{eqn:sol-repr}
 \rho_h^n = \tau \sum_{i=1}^n E_{h,\tau}^{n-i} [A_h(q^\dag)-A_h(\mathcal{I}_hq^\dag)]U_h^i(\mathcal{I}_hq^\dag),
\end{equation}
where $E_{h,\tau}^n$ is the fully discrete solution operator, which satisfies that for all $ v_h\in X_h$ \cite{JinLiZhou:nonlinear},
\begin{align*}
 \|  E_{h,\tau}^n v_h\|_{ {L^2}(\Omega) } & =   \|  A_h(q^\dag)^{\frac12}E_{h,\tau}^n (A_h(q^\dag)^{-\frac12}v_h)\|_{ {L^2}(\Omega) }\\
 &\le  ct_{n+1}^{-1+\frac{\alpha}{2}} \|A_h(q^\dag)^{-\frac12}v_h\|_{ {L^2}(\Omega) }
 \le c t_{n+1}^{-1+\frac{\alpha}{2}} \| v_h \|_{H^{-1}(\Omega)}.
\end{align*}
It follows from this estimate and the solution representation \eqref{eqn:sol-repr} that
\begin{equation*}
 \| \rho_h^n \|_{L^2(\Omega)} \le c\tau\sum_{i=1}^n t_n^{-1+\frac{\alpha}{2}} \|[A_h(\mathcal{I}_hq^\dag)-A_h(q^\dag)]U_h^n(\mathcal{I}_hq^\dag)\|_{H^{-1}(\Omega)}.
\end{equation*}
Further, the definitions of $P_h$ and $A_h$ and the $H^1(\Omega)$-stability of $P_h$ yield
\begin{equation*}
\begin{split}
 & \|[A_h(\mathcal{I}_hq^\dag)-A_h(q^\dag)]U_h^n(\mathcal{I}_hq^\dag)\|_{H^{-1}(\Omega)}
 = \sup_{v\in {\dot H}^1}\frac{\langle[A_h(\mathcal{I}_hq^\dag)-A_h(q^\dag)]U_h^n(\mathcal{I}_hq^\dag),v\rangle}{\| v \|_{{\dot{H}^1(\Omega)}}}\\
= & \sup_{v\in {\dot H}^1}\frac{\langle(q^\dag-\mathcal{I}_h q^\dag)\nabla U_h^n(\mathcal{I}_hq^\dag),\nabla P_hv\rangle}{\|v\|_{{\dot{H}^1(\Omega)}}}
\le  c h^2\|q^\dag\|_{W^{2,\infty}(\Omega)} \| \nabla U_h^n(\mathcal{I}_hq^\dag)  \|_{L^2(\Omega)},
\end{split}
\end{equation*}
since $q\in W^{2,\infty}(\Omega)$ by \cref{ass:data2}(i)
and \eqref{eqn:int-err-inf}. Thus,
$\| \rho_h^n \|_{L^2(\Omega)} \le ch^2 \tau\sum_{i=1}^n t_n^{-1+\frac{\alpha}{2}}\le ch^2\int_0^Tt^{-1+\frac{\alpha}{2}}\d t\leq ch^2$.
This and the triangle inequality completes the proof of the lemma.
\end{proof}

Next we give an error estimate on the CQ approximation of the fractional derivative. The proof is
similar to \cite[Lemma 4.2]{JinZhou:iterative}, and given in Appendix \ref{app:deriv} for completeness.
\begin{lemma}\label{lem:deriv-approx}
Let $q^\dag$ be the exact diffusion coefficient and $u\equiv u(q^\dag)$ be the solution to problem \eqref{eqn:var}.
Then under \cref{ass:data2}, there holds
\begin{equation*}
\begin{split}
  \| \bar \partial_\tau^\alpha (u(t_n) - u_0) - \partial_t^\alpha (u(t_n) - u_0) \|_{L^2(\Omega)}
           \le c\tau t_n^{-1}.
\end{split}
\end{equation*}
\end{lemma}

The next lemma gives a quadrature error estimate.
\begin{lemma}\label{lem:quadrature}
Let $q^\dag$ be the exact diffusion coefficient and $u\equiv u(q^\dag)$ the corresponding solution to problem \eqref{eqn:var}. Then
under  \cref{ass:data2},
\begin{equation*}
  \sum_{n=0}^Na_i\|u(t_n)-z^\delta(t_n)\|^2_{L^2(\Omega)} \leq c(\delta^2 + \tau^{1+\alpha}).
\end{equation*}
\end{lemma}
\begin{proof}
Let $g(t)=z^\delta(t)-u(t)$. By the regularity estimate \eqref{sol-reg} and Assumption
\ref{ass:data2},
\begin{equation}\label{eqn:bdd-g}
\|g\|_{C([0,T];L^2(\Omega))}\leq c, \quad \|g'(t)\|_{L^2(\Omega)} \leq c t^{\alpha-1} \quad \mbox{and}\quad
\|g''(t)\|_{L^2(\Omega)}\leq ct^{\alpha-2}.
\end{equation}
By the triangle inequality, we have
\begin{align*}
  & \Big|\tau \sum_{n=0}^Na_i \|  g(t_n) \|_{L^2(\Omega)}^2 - \sum_{n=1}^N \int_{t_{n-1}}^{t_n}\| g(t)\|_{L^2(\Omega)}^2\d t\Big|\\
\leq & \sum_{n=1}^N \Big| \int_{t_{n-1}}^{t_n}\|g(t)\|_{L^2(\Omega)}^2\d t - \frac{\tau}{2}\big(\|g(t_{n-1})\|_{L^2(\Omega)}^2+\|g(t_n)\|_{L^2(\Omega)}^2\big)\Big|:=\sum_{n=1}^N {\rm I}_n.
\end{align*}
Next we analyze  the two cases $n=1$ and $n>1$ separately. First, for the case $n=1$,
\begin{equation*}
  {\rm I}_1 \leq \Big|\int_{0}^\tau(\|g(t)\|_{L^2(\Omega)}^2 - \|g(t_0)\|^2_{L^2(\Omega)})\d t\Big| + \Big|\int_{0}^\tau(\|g(t)\|_{L^2(\Omega)}^2 - \|g(\tau)\|^2_{L^2(\Omega)})\d t\Big|:={\rm I}_{1,0} + {\rm I}_{1,1}.
\end{equation*}
Using  \eqref{eqn:bdd-g}, the term ${\rm I}_{1,0}$ can be bounded by
\begin{align*}
  {\rm I}_{1,0} &   \leq c\| g(t)\|_{C([0,\tau];L^2(\Omega))} \int_{0}^\tau\|g(0)-g(t)\|_{L^2(\Omega)}\d t
  \leq  c\tau \int^\tau_0\|g'(s) \|_{L^2(\Omega)}\d s \leq c\tau^{1+\alpha}.
\end{align*}
Similarly, we can deduce ${\rm I}_{1,1}\leq c\tau^{1+\alpha}$. Further, for the case $n>1$, $g(t)$ is smooth,
and thus by standard interpolation error estimates, for some $\xi_n\in [t_{n-1},t_n]$, there holds
$
{\rm I}_n \leq c\tau^2 \int_{t_{n-1}}^{t_n} \big|\frac{\d^2}{\d t^2}\|g(t)\|^2_{L^2(\Omega)}|_{t=\xi_n}\big|\d t.
$
By the bounds in \eqref{eqn:bdd-g},
$  \big|\frac{\d^2}{\d t^2}\|g(\xi_n)\|^2_{L^2(\Omega)}\big| \leq 2(\|g'(\xi_n)\|_{L^2(\Omega)}^2+\|g(\xi_n)\|_{L^2(\Omega)}\|g''(\xi_n)\|_{L^2(\Omega)})\leq ct_{n-1}^{\alpha-2}.
$
The last two estimates together imply
\begin{equation*}
 \sum_{n=2}^N {\rm I}_n \leq c\tau^3\sum_{n=2}^N t_{n-1}^{\alpha-2}\leq c\tau^{1+\alpha}.
\end{equation*}
Then the assertion follows from the triangle inequality and the definition of the noise level.
\end{proof}
\begin{remark}
One can only obtain an $O(\tau+\delta^2)$ rate the discrete objective function $J_{\gamma,h,\tau}$ in
\eqref{eqn:ob-disc}. The $\alpha$ exponent in  \cref{lem:quadrature} reflects the limited temporal
smoothing property of the solution $u(t)$: the larger the fractional order $\alpha$ is, the smoother
in time the solution $u(t)$ becomes and the faster the quadrature error decays.
\end{remark}

The next result gives \textit{a priori} bounds on $q_h^*$ and
the approximation $U_h^n(q_h^*)$. This result will play a crucial role in the analysis below.
\begin{lemma}\label{lem:err-2}
Let $q^\dag$ be the exact coefficient and $u\equiv u(q^\dag)$ the solution to problem \eqref{eqn:var}.
Let $q_h^*\in \Uad_h$ be the solution to problem \eqref{eqn:ob-disc-1}--\eqref{eqn:fully-1}, and
$\{U_h^n(q_h^*)\}_{n=1}^N$ the fully discrete solution to problem \eqref{eqn:fully}. Then under
 \cref{ass:data2}, with $\ell_h=|\log h|$, there holds
\begin{equation*}
  \tau \sum_{n=1}^N \| U_h^n(q_h^*) - u(t_n) \|_{L^2(\Omega)}^2 + \gamma \| \nabla q_h^* \|_{L^2(\Omega)}^2
  \le c  (\tau^{1+\alpha}  + h^4\ell_h^2 + \delta^2+\gamma).
\end{equation*}
\end{lemma}
\begin{proof}
By the minimizing property of $q_h^*\in \Uad_h$ and $\mathcal{I}_h q^\dag \in \Uad_h$, we deduce
$J_{\gamma,h,\tau}(q_h^*)\leq J_{\gamma,h,\tau}(\mathcal{I}_hq^\dag).$
By the triangle inequality, we derive
\begin{align*}
\tau \sum_{n=1}^N \|U_h^n(q_h^*) -  u(t_n) \|_{L^2(\Omega)}^2
 &\le  c \tau \sum_{n=1}^N \|U_h^n(q_h^*) - z^\delta(t_n)\|_{L^2(\Omega)}^2
  + c \tau \sum_{n=0}^N a_n\| z^\delta(t_n)-u(t_n) \|_{L^2(\Omega)}^2 .
\end{align*}
These two inequalities and \cref{lem:quadrature} imply
\begin{align*}
   & \tau \sum_{n=1}^N \| U_h^n(q_h^*) - u(t_n) \|_{L^2(\Omega)}^2 + \gamma \| \nabla q_h^* \|_{L^2(\Omega)}^2 \\
  \leq & c\tau \sum_{n=1}^N \| U_h^n(\mathcal{I}_h q^\dag) - z^\delta(t_n)\|_{L^2(\Omega)}^2 + c\gamma\| \nabla \mathcal{I}_hq^\dag\|_{L^2(\Omega)}^2 + c(\delta^2+\tau^{1+\alpha}).
\end{align*}
Since $q^\dag\in W^{1,\infty}(\Omega)$ by Assumption
\ref{ass:data2}, $\|\nabla \mathcal{I}_hq^\dag\|_{L^2(\Omega)}\leq c$, cf. \eqref{eqn:int-err-inf}.
Further, by  \cref{lem:err-1}, we have
\begin{align*}
 \|U_h^n(\mathcal{I}_hq^\dag) - z^\delta(t_n)\|_{L^2(\Omega)}^2 &\le 2\|U_h^n(\mathcal{I}_hq^\dag) - u(t_n)\|_{L^2(\Omega)}^2+2\|u(t_n) - z^\delta(t_n)\|_{L^2(\Omega)}^2\\
& \le c(\tau t_n^{\alpha-1} + h^2\ell_h)^2 + c\|u(t_n) - z^\delta(t_n)\|_{L^2(\Omega)}^2,
\end{align*}
Consequently,
\begin{align*}
   \tau \sum_{n=1}^N \|  \nabla( U_h^n(\mathcal{I}_hq^\dag) - &z^\delta(t_n) ) \|_{L^2(\Omega)}^2 \le c \tau \sum_{n=1}^N ( t_n^{\alpha-1} \tau + h^2\ell_h)^2 + c\sum_{n=0}^Na_n\|u(t_n) - z^\delta(t_n)\|_{L^2(\Omega)}^2\\
   \leq & c\tau^3 \sum_{n=1}^N t_n^{\alpha-2} + c h^4\ell_h^2+ c(\tau^{1+\alpha} +\delta^2)
   \leq  c(\tau^{1+\alpha} + h^4\ell_h^2 + \delta^2 ).
\end{align*}
Combining the preceding estimates completes the proof of the lemma.
\end{proof}

We shall also need the following lemma on backward Euler CQ.
\begin{lemma}\label{lem:CQ}
Let $q^\dag$ be the exact coefficient, and $u\equiv u(q^\dag)$ the corresponding solution to problem \eqref{eqn:fde}.
Then for $\varphi^m=\frac{q^\dag-q_h^*}{q^\dag}u(t_m)$, and any $\epsilon\in(0,\min(\frac12,1-\alpha))$, there holds
\begin{align*}
  \| \tau^{-\alpha} \sum_{n=j}^m b_{n-j}^{(\alpha)} P_h(\fy^n - \fy^m) \|_{L^2(\Omega)} & \leq c_{T,\epsilon} t_j^{-\epsilon}.
\end{align*}
\end{lemma}
\begin{proof}
By the associativity of CQ from  \eqref{eqn:cq-ass}, i.e., $\bar\partial_\tau^\alpha\varphi^n
=\bar\partial_\tau^{\alpha-1}\bar\partial_\tau \varphi^n$, if $\varphi^0=0$,
\begin{align*}
   {\rm I}:=\tau^{-\alpha} \sum_{n=j}^m b_{n-j}^{(\alpha)} P_h(\fy^n - \fy^m) & = \tau^{1-\alpha} \sum_{n=j}^m b_{n-j}^{(\alpha-1)}\frac{P_h\fy^n-P_h\fy^{n+1}}{\tau}.
\end{align*}
Thus, the $L^2(\Omega)$-stability of $P_h$, the bound on $|b_j^{(\alpha-1)}|\leq c(j+1)^{-\alpha}$ and \eqref{sol-reg} imply
\begin{align*}
   \|{\rm I}\|_{L^2(\Omega)}
   \leq &\tau^{1-\alpha} \sum_{n=j}^m |b_{n-j}^{(\alpha-1)}| \|\frac{\fy^n - \fy^{n+1}}{\tau}\|_{L^2(\Omega)}
 \leq  c\tau^{1-\alpha} \sum_{n=j}^m (n-j+1)^{-\alpha}\|\fy'(\xi_n)\|_{L^2(\Omega)}\\
 \leq & c\tau^{1-\alpha}\sum_{n=j}^m (n-j+1)^{-\alpha}t_n^{\alpha-1}
 \leq c\int_{t_j}^{t_m} (s-t_j+\tau)^{-\alpha}s^{\alpha+\epsilon-1}\d s t_j^{-\epsilon} =: g(t_j) t_j^{-\epsilon}.
\end{align*}
where $\xi_n\in [t_n,t_{n+1}]$.
We claim that the integral $g(t_j)$ is decreasing in $t_j\in [\tau,t_m]$.
Indeed, for any $0<\bar t_1<\bar t_2\leq t_m$, by changing of variables, there holds
\begin{align*}
g(\bar t_1) &:= \int_{\bar t_1}^{t_m}  (s-\bar t_1+\tau)^{-\alpha} s^{\alpha+\epsilon-1} \,\d s \\
&= \int_{\bar t_1}^{t_m - (\bar t_2-\bar t_1)}  (s-\bar t_1+\tau)^{-\alpha} s^{\alpha+\epsilon-1} \,\d s + \int_{t_m - (\bar t_2-\bar t_1)}^{t_m}(s-\bar t_1+\tau)^{-\alpha} s^{\alpha+\epsilon-1} \,\d s \\
&\ge g(\bar t_2) + \int_{t_m - (\bar t_2-\bar t_1)}^{t_m}(s-\bar t_1+\tau)^{-\alpha} s^{\alpha+\epsilon-1} \,\d s \ge g(\bar t_2).
\end{align*}
Thus,
$\|{\rm I}\|_{L^2(\Omega)} \leq c t_j^{-\epsilon}\int_\tau^{t_m}(s+\tau)^{-\alpha}s^{\alpha+\epsilon-1} \d s\le
  c_\epsilon t_j^{-\epsilon}.$ This completes the proof of the lemma.
\end{proof}

The next theorem represents the main result of this section, i.e., error estimate of the numerical approximation
$q_h^*\in \Uad_h$ in a weighted $L^2(\Omega)$ norm, with the weight $q^\dag|\nabla u(t_n)  |^2 +( f(t_n) -\partial_t^\alpha u(t_n))u(t_n)$.
The proof relies crucially on the choice of the novel test function $\varphi=\frac{q^\dag-q_h^*}{q^\dag}u$.
\begin{theorem}\label{thm:error-q}
Let $q^\dag$ be the exact diffusion coefficient, $u\equiv u(q^\dag)$ the solution to problem \eqref{eqn:var},
and $q_h^*\in \Uad_h$ the solution to problem \eqref{eqn:ob-disc-1}--\eqref{eqn:fully-1}. Then
under \cref{ass:data2}, for $d=1,2$, with $\ell_h=|\log h|$ and $\eta= \tau^{\frac12+\frac\alpha2} + h^2\ell_h
+ \delta + \gamma^{\frac12}$, there holds
\begin{align*}
 &\tau^2 \sum_{m=1}^N \sum_{n=1}^m \int_\Omega \Big(\frac{q^\dag-q_h^*}{q^\dag}\Big)^2 \Big(q^\dag|\nabla u(t_n)  |^2 +( f(t_n) -\partial_t^\alpha u(t_n))u(t_n)\Big)\,\d x\\
 \le& c(h \gamma^{-1}\eta+ h\gamma^{-\frac12} +  h^{-1}\gamma^{-\frac12}\eta )\eta.
\end{align*}
\end{theorem}
\begin{proof}
For any test function $\fy$ to be specified below, we have the splitting
\begin{align*}
 ((q^\dag-q_h^*)\nabla u(t_n),\nabla\fy)& =((q^\dag-q_h^*)\nabla u(t_n),\nabla(\fy-P_h\fy))+ (q^\dag\nabla u(t_n)-q_h^*\nabla u(t_n),\nabla P_h\fy).
\end{align*}
Thus, applying integration by parts to the first term leads to
 \begin{align}\label{eqn:sp-01}
 ((q^\dag-q_h^*)\nabla u(t_n),\nabla\fy)&=-(\nabla\cdot((q^\dag-q_h^*)\nabla u(t_n)), \fy-P_h\fy ) + (q_h^*\nabla (U_h^n(q_h^*) - u(t_n)),\nabla P_h\fy) \nonumber\\
   & \quad + (q^\dag\nabla u(t_n) - q_h^*\nabla U_h^n(q_h^*),\nabla P_h\fy) =\sum_{i=1}^3{\rm I}_i^n.
\end{align}
Next we bound the three terms. Direct computation with the triangle inequality gives
\begin{align*}
 \| \nabla\cdot((q^\dag-q_h^*)\nabla u(t_n))\|_{L^2(\Omega)} \le & \| \nabla q^\dag\|_{L^\infty(\Omega)}  \| \nabla u(t_n) \|_{L^2(\Omega)}
 +\| q^\dag-q_h^*\|_{L^\infty(\Omega)}\| \Delta u(t_n) \|_{L^2(\Omega)}\\
   & +  \| \nabla q_h^*\|_{L^2(\Omega)}  \| \nabla u(t_n) \|_{L^\infty(\Omega)}.
\end{align*}
In view of the regularity estimate \eqref{sol-reg}, we derive
\begin{equation*}
\begin{split}
 \| \nabla\cdot (q^\dag-q_h^*)\nabla u(t_n) \|_{L^2(\Omega)} &\le c +  \| \nabla q_h^*\|_{L^2(\Omega)}  \| \nabla u(t_n) \|_{L^\infty(\Omega)}\\
 &\le c (1+t_n^{\min(0,1-\frac{d}{2}-\epsilon)\frac{\alpha}{2} }\| \nabla q_h \|_{L^2(\Omega)} ),
\end{split}
\end{equation*}
where the second line is due to Sobolev embedding $\| \nabla u \|_{L^\infty(\Omega)} \le c \|  u \|_{H^s(\Omega)}$ with
$s>\frac{d}2+1$ (by the convexity of the domain and elliptic regularity \cite[Corollary 19.7, p. 166]{Dauge:1988}).
This and the Cauchy-Schwarz inequality imply that the first term ${\rm I}_1^n$ is bounded by
\begin{equation*}
 |{\rm I}_1^n| \le  c(1+\| \nabla q_h \|_{L^2(\Omega)} ) \|  \fy-P_h\fy \|_{L^2(\Omega)}.
\end{equation*}
Now we choose the test function $\fy$ to be $\fy\equiv \fy^n=\frac{q^\dag-q_h^*}{q^\dag} u(t_n) \in H_0^1(\Omega)$,
and then straightforward computation gives
$
\nabla \fy^n = \big(q^{\dag-1} \nabla(q^\dag-q_h^*) - q^{\dag-2}(q^\dag-q_h^*)\nabla q^\dag\big) u(t_n) + q^{\dag-1}(q^\dag-q_h^*)\nabla u(t_n).
$
By the box constraint of $\Uad$ and the regularity estimate \eqref{sol-reg}, we have
\begin{equation*}
  \|\nabla\fy^n\|_{L^2(\Omega)}\le c\Big[(1+\|\nabla q_h^*\|_{L^2(\Omega)})\|u(t_n)\|_{L^\infty(\Omega)}+ \| \nabla u(t_n)\|_{L^2(\Omega)}\Big]\le c(1+\|\nabla q_h^*\|_{L^2(\Omega)}),
\end{equation*}
and the approximation property of the projection operator $P_h$ implies
$
  \|\fy^n -P_h\fy^n\|_{L^2(\Omega)} \le ch\| \nabla \fy^n\|_{L^2(\Omega)} \le ch(1+\|\nabla q_h^*\|_{L^2(\Omega)}).
$
Thus, by \cref{lem:err-2}, the term ${\rm I}_1^n$ is bounded by
\begin{align*}
  |{\rm I}_1^n| & \le  cht_n^{\min(0,1-\frac{d}{2}-\epsilon)\frac{\alpha}{2}}(1+\| \nabla q_h ^*\|_{L^2(\Omega)} )^2\\
  &\le  c t_n^{\min(0,1-\frac{d}{2}-\epsilon)\frac{\alpha}{2}}  h(1+\gamma^{-1}\eta^2) \le  c t_n^{\min(0,1-\frac{d}{2}-\epsilon)\frac{\alpha}{2}}  h \gamma^{-1}\eta^2,
\end{align*}
which together with the trivial inequality
$\tau \sum_{n=1}^N t_n^{\min(0,1-\frac d2-\epsilon)\frac{\alpha}{2}} \leq c$
implies
\begin{equation}\label{eqn:I1}
 \tau \sum_{n=1}^N  {\rm I}_1^n  \le   ch \gamma^{-1}\eta^2  .
\end{equation}
For the term ${\rm I}_2^n$, by the triangle inequality, inverse inequality, $ H^1(\Omega)$ stability of $P_h$,
 we have
\begin{align*}
\|  \nabla(u(t_n) - U_h^n(q_h^*)) \|_{L^2(\Omega)} & \leq \|  \nabla(u(t_n) - P_hu(t_n) ) \|_{L^2(\Omega)} + h^{-1}\|  P_h u(t_n)  - U_h^n(q_h^*)  \|_{L^2(\Omega)}\\
  & \leq c(h + h^{-1}\| u(t_n) - U_h^n(q_h^*)   \|_{L^2(\Omega)},
\end{align*}
and consequently, the Cauchy-Schwarz inequality and \cref{lem:err-2} imply
\begin{align}\label{eqn:I2}
  \tau \sum_{n=1}^N  {\rm I}_2^n & \le  \tau \sum_{n=1}^N \|  \nabla(u(t_n) - U_h^n(q_h^*)) \|_{L^2(\Omega)} \|  \nabla \fy^n \|_{L^2(\Omega)} \nonumber\\
  &\le c \Big(h + h^{-1}\Big(\tau \sum_{n=1}^N \|  u(t_n) - U_h^n(q_h^*)  \|_{L^2(\Omega)}^2\Big)^{\frac12}\Big) (1+\|\nabla q_h^* \|_{L^2(\Omega)})\nonumber\\
  &\le c (h \gamma^{-\frac12} + h^{-1} \gamma^{-\frac12} \eta ) \eta.
\end{align}
Next we bound the third term ${\rm I}_3^n$. It follows directly from \eqref{eqn:var} and \eqref{eqn:fully} that
\begin{align*}
 {\rm I}_3^n &=   (q^\dag\nabla u(t_n) - q_h^*\nabla U_h^n(q_h^*),\nabla P_h\fy^n ) \\
   &= (\bar \partial_\tau^\alpha (U_h^n(q_h^*)-U_h^0) - \partial_t^\alpha (u(t_n)-u_0) , P_h\fy^n )\\
 &=(\bar \partial_\tau^\alpha [(U_h^n(q_h^*) - U_h^0)-  (u(t_n)-u_0)],  P_h\fy^n )  \\
   & \quad + ( \bar \partial_\tau^\alpha (u(t_n) - u_0) - \partial_t^\alpha (u(t_n) - u_0) , P_h\fy^n )=:  {\rm I}_{3,1}^n + {\rm I}_{3,2}^n.
\end{align*}
It remains to bound the two terms ${\rm I}_{3,1}^n$ and ${\rm I}_{3,2}^n$ separately.
By  \cref{lem:deriv-approx}, there holds
\begin{equation*}
|{\rm  I}_{3,2}^n| \le  \| \bar \partial_\tau^\alpha (u(t_n) - u_0) - \partial_t^\alpha (u(t_n) - u_0) \|_{L^2(\Omega)}  \| P_h\fy^n \|_{L^2(\Omega)}
           \le c\tau t_n^{-1}, \quad n=1,2,\ldots, N.
\end{equation*}
Consequently,
\begin{equation*}
 |\tau^2\sum_{m=1}^N \sum_{n=1}^m {\rm I}_{3,2}^n| \le  c \tau^3  \sum_{m=1}^N \sum_{n=1}^m t_n^{-1} \le c \tau \log(1+t_N/\tau).
\end{equation*}
It remains to bound the term ${\rm I}_{3,1}^n$. Since $U_h^0(q_h^*)=U_h^0$ and $u(0)=u_0$, straightforward computation with
summation by parts yields
\begin{align*}
\tau \sum_{n=1}^m {\rm I}_{3,1}^n=& \tau \sum_{n=0}^m  (\bar \partial_\tau^\alpha [(U_h^n(q_h^*) - U_h^0)-  (u(t_n)-u_0)],  P_h\fy^n)\\
= & \tau \sum_{j=0}^m  ([ (U_h^j(q_h^*) - U_h^0)-  (u(t_j)-u_0)],  \tau^{-\alpha} \sum_{n=j}^m b_{n-j}^{(\alpha)} P_h\fy^n ).
\end{align*}
Next we appeal to the splitting
\begin{equation*}
  \tau^{-\alpha} \sum_{n=j}^m b_{n-j}^{(\alpha)} P_h\fy^n = \tau^{-\alpha} \sum_{n=j}^m b_{n-j}^{(\alpha)} P_h(\fy^n-\fy^m) + \tau^{-\alpha} \sum_{n=j}^m b_{n-j}^{(\alpha)} P_h\fy^m
: = {\rm IV}_{j,m}^1 +  {\rm IV}_{j,m}^2.
\end{equation*}
By \cref{lem:CQ-0}, the sum ${\rm IV}_{j,m}^2$ satisfies
\begin{align*}
 \| {\rm IV}_{j,m}^2\|_{L^2(\Omega)} & \le c \| \fy^m \|_{L^2(\Omega)} \Big(\tau^{-\alpha} \sum_{n=0}^{m-j} b_{n}^{(\alpha)}\Big)
   \le c t_{m-j+1}^{-\alpha} \| \fy^m \|_{L^2(\Omega)}\le c t_{m-j+1}^{-\alpha} ,
\end{align*}
since $\|\fy^m\|_{L^2(\Omega)}\leq c$.
Then  \cref{lem:err-2} and Cauchy-Schwarz inequality imply
\begin{align*}
\tau^2 \sum_{m=1}^N \sum_{j=1}^m \|U_h^j(q_h^*)  -  u(t_j)  \|_{L^2(\Omega)}  \|{\rm IV}_{j,m}^2\|_{L^2(\Omega)}
 \le& c \tau^2 \sum_{j=1}^N \sum_{m=j}^N \|U_h^j(q_h^*)  -  u(t_j)  \|_{L^2(\Omega)}  t_{m-j+1}^{-\alpha}\\
 \le&  c \Big(\tau  \sum_{j=1}^N \|U_h^j(q_h^*)  -  u(t_j)  \|_{L^2(\Omega)}^2\Big)^{\frac12}
 \le c\eta,
\end{align*}
where the second inequality is due to $\tau \sum_{m=j}^N t_{m-j+1}^{-\alpha} \leq ct_{N-j+1}^{1-\alpha}$.
Similarly, by \cref{lem:CQ},
\begin{align*}
&\tau^2 \sum_{m=1}^N \sum_{j=1}^m \|U_h^j(q_h^*)-u(t_j)\|_{L^2(\Omega)}\|{\rm IV}_{j,m}^1\|_{L^2(\Omega)}
\le c\tau^2 \sum_{m=1}^N \sum_{j=1}^m \|U_h^j(q_h^*)-u(t_j)\|_{L^2(\Omega)} t_j^{-\epsilon} \\
\le&   c\tau \sum_{j=1}^N \| u_h^j(q_h)  -  u(t_j;q)  \|_{L^2(\Omega)} t_j^{-\epsilon}
\le  c  \Big(\tau  \sum_{j=1}^N \|U_h^j(q_h^*)-u(t_j)\|_{L^2(\Omega)}^2\Big)^{\frac12} \le c \eta.
\end{align*}
These two estimates and the triangle inequality lead to
\begin{equation}\label{eqn:I3}
\Big|\tau^2 \sum_{m=1}^N \sum_{n=1}^m(\bar\partial_\tau^\alpha [(U_h^n(q_h^*)-u_h^0)-(u(t_n)-u_0)],P_h\fy^n )\Big|\le c\eta .
\end{equation}
The three estimates \eqref{eqn:I1}, \eqref{eqn:I2} and \eqref{eqn:I3} together imply
\begin{align*}
 \tau^2\sum_{m=1}^N \sum_{n=1}^m((q^\dag-q_h^*)\nabla u(t_n),\nabla\fy^n)
  \le  c(h \gamma^{-1}\eta + \gamma^{-\frac12}\eta+ h^{-1}\gamma^{-\frac12}\eta)\eta.
\end{align*}
Finally, this and the identity
\begin{equation*}
((q^\dag-q_h^*)\nabla u(t_n),\nabla\fy^n) = \frac{1}{2}\int_\Omega \Big(\frac{q^\dag-q_h^*}{q^\dag}\Big)^2\Big( q^\dag|\nabla u(t_n)  |^2 +(f(t_n)-\partial_t^\alpha u(t_n))u(t_n)\Big)\,\d x
\end{equation*}
lead immediately to the desired assertion. This completes the proof of the theorem.
\end{proof}

\begin{remark}
The restriction on $d=1,2$ is due to limited regularity pickup on general convex polyhedral domains, in order to
ensure $\| \nabla u \|_{L^\infty(\Omega)} \le c \|  u \|_{H^s(\Omega)}\leq c$. The result holds also for a
polyhedral domain in $\mathbb{R}^3$ with suitable conditions \cite[Theorem 4, p. 18]{Dauge:2008}.
{One possible strategy to remove the restriction is to use maximal $L^p(\Omega)$ regularity
\cite{JinLiZhou:2018nm}, instead of
the Hilbert space $H^s(\Omega)$.} Further, it is worth noting that the proof relies heavily on the
discrete ``integration by parts'' formula for convolution quadrature when bounding the term ${\rm I}_{3,1}$, which is
valid only for the whole interval $[0,T]$ and represents the main obstacle in extending the analysis
to the case of partial data, e.g., terminal observation.
\end{remark}

The next result is an immediate corollary of  \cref{thm:error-q}.
\begin{corollary}\label{cor:err-q}
Let $q^\dag$ be the exact diffusion coefficient, $u\equiv u(q^\dag)$ the solution to problem \eqref{eqn:var},
and $q_h^*\in \mathcal{A}_h$ the solution to problem \eqref{eqn:ob-disc-1}--\eqref{eqn:fully-1}. Then
under \cref{ass:data2}, for $d=1,2$, there holds {\rm(}with
$\eta= \tau^{\frac12+\frac{\alpha}{2}} + h^2\ell_h + \delta + \gamma^{\frac12}${\rm)}
\begin{align*}
  &\int_0^T\int_0^t \int_\Omega \Big(\frac{q^\dag-q_h^*}{q^\dag}\Big)^2 \Big(q^\dag|\nabla u(s)  |^2 +
  (f(s) -\partial_s^\alpha u(s))u(s)\Big)\,\d x\d s \d t\\
  \le & c(h \gamma^{-1}\eta+ h\gamma^{-\frac12} +  h^{-1}\gamma^{-\frac12}\eta)\eta.
\end{align*}
\end{corollary}
\begin{proof}
In view of \cref{thm:error-q}, it suffices to bound the quadrature error:
\begin{align*}
  & \left|\int_0^T\int_0^t |\nabla u(s)|^2\d s\d t - \tau^2 \sum_{m=1}^N \sum_{n=1}^m|\nabla u(t_n)  |^2\right| \\
    & + \left|\int_0^T\int_0^t(f(s)-\partial_s^\alpha u(s))u(s)\d s \d t-
  \tau^2 \sum_{m=1}^N \sum_{n=1}^m  ( f(t_n) -\partial_t^\alpha u(t_n))u(t_n)\right|
  := {\rm I} + {\rm II}.
\end{align*}
It remains to bound the two terms ${\rm I}$ and ${\rm II}$.  For the first term,
\begin{align*}
  {\rm I} & \leq \left|\sum_{m=1}^N\big(\int_{t_{m-1}}^{t_m} \int_0^{t_m} |\nabla u(s)|^2\d s\d t - \tau^2 \sum_{n=1}^m  |\nabla u(t_n)  |^2\big)\right|\\
          & + \left|\sum_{m=1}^N\int_{t_{m-1}}^{t_m} \int_{\max(t,t_{m-1})}^{t_m} |\nabla u(s)|^2\d s\d t \right| \\
          & \leq \tau \sum_{m=1}^N \underbrace{\Big|\int_0^{t_m}|\nabla u(s)|^2\d s -\tau\sum_{n=1}^m  |\nabla u(t_n)  |^2\Big|}_{{\rm I}_m} + \tau\sum_{m=1}^N \int_{t_{m-1}}^{t_m} |\nabla u(s)|^2\d s.
\end{align*}
By the regularity estimate \eqref{sol-reg}, $\|\nabla u'(s)\|_{L^2(\Omega)}\leq cs^{\frac{\alpha}{2}-1}$ and $\|\nabla u(t)\|_{C([0,T];L^2(\Omega))}\leq c$. Clearly $\tau\sum_{m=1}^N \int_{t_{m-1}}^{t_m} |\nabla u(s)|^2\d s \leq c\tau$. Further,
\begin{align*}
  \int_\Omega {\rm I}_m\d x  & \leq \sum_{n=1}^m\int_{t_{m-1}}^{t_m}\|\nabla (u(s)+u(t_n))\|_{L^2(\Omega)}\|\nabla (u(s)- u(t_n))\|_{L^2(\Omega)} \d s\\
    & \leq  c\|\nabla u\|_{C([0,t_m];L^2(\Omega))}\sum_{n=1}^m\int_{t_{m-1}}^{t_m}\|\nabla \int_s^{t_n}u'(\zeta)\d \zeta\|_{L^2(\Omega)} \d s \\
    & \leq c\|\nabla u\|_{C([0,t_m];L^2(\Omega))}\tau \int_0^{t_m}s^{\frac{\alpha}{2}-1}\d s \leq c\tau.
\end{align*}
The preceding two estimates imply $\int_\Omega {\rm I}\d x \leq c\tau$. The term ${\rm II}$
can be bounded similarly as $\int_\Omega{\rm II}\d x \leq c\tau|\ln \tau|$. Indeed, under  \cref{ass:data2}(i),
the estimate \eqref{sol-reg} and \eqref{eqn:fde}, we have $\|\partial_t^\alpha u\|_{L^2(\Omega)}
\leq c$ and $\|(\partial_t^\alpha u)'(t)\|_{L^2(\Omega)}\leq ct^{-1}$, and thus $g(t)\equiv
\partial_t^\alpha u(t)-f(t)$ satisfies $\|g(t)\|_{L^2(\Omega)}\leq c$ and $\|g'(t)\|_{L^2(\Omega)}\leq ct^{-1}$.
Then repeating the argument completes the proof.
\end{proof}

\begin{remark}{
There has been much interest in deriving error bounds on the Galerkin approximation $q^*_h$ in the usual $L^2(\Omega)$ or
Sobolev norm for nonlinear parameter identification problems. However, for the inverse conductivity problem in
either elliptic or parabolic case, such an estimate remains elusive, largely due to a lack of convexity of the regularized problem.
The error estimate given in Corollary \ref{cor:err-q} provides one possible route to derive an $L^2(\Omega)$ estimate. Indeed,
if the exact coefficient $q^\dag$ and the corresponding state $u\equiv u(q^\dag)$ satisfy
\begin{equation}\label{eqn:ass-uq}
 \int_0^T\int_0^t \Big(q^\dag|\nabla u(s)  |^2 +( f(s) -\partial_s^\alpha u(s))u(s)\Big)\d s\d t > {c} \qquad \text{a.e. }x \in{\Omega},
\end{equation}
the the usual $L^2$ estimate follows directly. In the classical parabolic case, similar structural conditions have
been assumed in the literature, e.g., the following characteristic condition \cite{TaiKarkkarinen:1995,Karkkainen:1997}:
$t^{-1}\int_0^t\nabla u(x,s)\d s\cdot \nu\geq \delta_0>0$ for all $(x,t)\in Q\equiv \Omega\times(0,T]$, where $\nu$ is
a constant vector, or \cite[Theorem 6.4]{WangZou:2010} $\alpha_0 |\int_0^t \nabla u (x,s)\d s|^2 + t\int_0^t (u'(x,s)
-f(x,s))\d s \geq 0$ a.e. $(x,t)\in Q$. Note that this latter condition is not positively homogeneous (with respect to
problem data). Next we comment on the condition \eqref{eqn:ass-uq}. If $f \equiv 0$ in $Q$, $u_0 > 0$ in $\Omega$, then
the maximum principle for the subdiffusion model \cite{LiuRundellYamamoto:2016} implies $u>0$ in $ Q$. Further,
$w=\partial_t^\alpha u$ satisfies $\Dal  w - \nabla\cdot(q^\dag\nabla w)  = \partial_t^\alpha f$ in $Q$, with initial
condition $w(0) = \nabla\cdot(q^\dag\nabla u_0)+ f(0)$ in $\Omega$ and boundary condition $w=0$ on $\partial\Omega\times(0,T]$
If $ \partial_t^\alpha f(t) \le 0$ and $\nabla\cdot(q^\dag\nabla u_0)+ f(0)\le 0$, then maximum principle implies
$\partial_t^\alpha u = w \le 0$ in $Q$. Further, if $f > 0$ in $ Q$, then $ f-\partial_t^\alpha u > 0$ in $ Q$, which
implies $ (f - \partial_t^\alpha u) u > 0$ in $Q$. Thus at least a weak version of condition \eqref{eqn:ass-uq} holds.
We leave further discussions on the condition \eqref{eqn:ass-uq} and its analogues to future work.}
\end{remark}

\begin{remark}\label{rem:rate}
\cref{thm:error-q} and Corollary \ref{cor:err-q} show that the convergence rate
is of order $O(\delta^\frac14)$ in the weighted norm, provided that $\gamma=O(h^4) =
O(\delta^2)=O(\tau^{1+\alpha})$. The error estimate in \cref{thm:error-q} and
Corollary \ref{cor:err-q} is expected to be sub-optimal, due to the presence of the factor
$h^{-1}$, which arises from the use of inverse inequality in \eqref{eqn:I2}. It remains unclear
how to achieve optimality, even in the standard parabolic case \cite{WangZou:2010}.
\end{remark}

\section{Numerical results and discussions} \label{sec:numer}

Now we present numerical results to illustrate the fully discrete scheme \eqref{eqn:ob-disc}--\eqref{eqn:fully} with one-
and two-dimensional examples, with the measurement $z^\delta$ over the time interval $[T_0,T]$ (by a straightforward
adaptation of the formulation; see Remark \ref{rmk:alternative}), with $T$ fixed at $1$. Throughout, the
corresponding discrete problem is solved by the conjugate gradient (CG) method \cite{Alifanov:1995}, with the gradient
computed using the standard adjoint technique. Unless otherwise stated, the lower and upper bounds in the admissible
set $\mathcal{A}$ are taken to be $c_0=0.5$ and $c_1=5$, respectively, and are enforced by a projection step after
each CG iteration. The minimization method converges generally within tens of iterations. The noisy data $z^\delta$ is generated by
\begin{equation*}
  z^\delta(x,t) = u(q^\dag)(x,t) + \epsilon\sup_{(x,t)\in\Omega\times [T_0,T]}|u(x,t)|\xi(x,t),\quad (x,t)\in \Omega\times[T_0,T],
\end{equation*}
where $\xi(x,t)$ follows the standard Gaussian distribution, and $\epsilon\geq 0$ denotes the (relative) noise level.
The noisy data $z^\delta$ is first generated on a fine spatial-temporal mesh and then interpolated to a coarse spatial/
temporal mesh for the inversion step.
The scalar $\gamma$ in the functional $J_\gamma$ plays an important role in determining the accuracy
of the reconstructions, but it is notoriously challenging to choose (see e.g., \cite{ItoJin:2015}). In our experiments, its value is determined
by a trial and error manner, first for the fractional order $\alpha=0.50$, and then used for the cases
$\alpha=0.25$ and $\alpha=0.75$, which might be suboptimal but works reasonably well in practice.

\subsection{Numerical results in one spatial dimension}
First we present numerical results for two examples on unit interval $\Omega=(0,1)$. The reference data $u(q^\dag)$ is computed with
a mesh size $h=1/400$ and time step size $\tau=1/2048$, and the inversion step is carried out with a mesh size $h=1/200$ and time
step size $\tau=1/1024$, unless otherwise specified.

The first example has a smooth exact coefficient $q^\dag$, and the problem is homogeneous.
\begin{example}\label{exam:1dsmooth}
$u_0=x(1-x)$, $f\equiv 0$, $q^\dag=2+\sin(2\pi x)$.
\end{example}

First, we let $T_0=0.75$ and study how the reconstruction error changes with respect to different parameters.
The numerical results for the example with different noise levels $\epsilon$, and fixed $h$ and $\tau$,
are summarized in Table \ref{tab:exam1}.
The chosen $\gamma$ is relatively small, since the magnitude of the
exact data $u(q^\dag)$ is actually very small: for example, upon convergence, the functional value $J_{\gamma,h,\tau}
(q_h^*)$ is about $O(10^{-12})$ for exact data and about $O(10^{-9})$ for $\epsilon=\text{1.00e-2}$.
Clearly, the $L^2(\Omega)$ error $e_q$ of the reconstruction $q_h^*$, i.e.,
$e_q=\|q^\dag-q_h^*\|_{L^2(\Omega)}$, decreases steadily as the noise level $\epsilon$ tends to zero (Note that
even at $\epsilon=0$, the reconstruction error $e_q$ is nonzero due to the presence of discretization errors). The convergence is
consistently observed for all three fractional orders. Interestingly, for a fixed noise level $\epsilon$, as
the fractional order $\alpha$ increases from $0.25$ to $0.75$, the reconstruction error tends to deteriorate slightly.
It might be related to the fact that for homogeneous subdiffusion, the smaller $\alpha$ is,
the quicker the state $u(t)$ approaches a ``quasi''-steady state; Then the inverse problem reduces
to the elliptic counterpart, i.e., $-\nabla\cdot(q\nabla u)=f$, which is known to be beneficial for numerical reconstruction \cite{JinRundell:2015}.
However, the precise mechanism remains to be ascertained. We refer to Fig. \ref{fig:recon-exam1} for
exemplary reconstructions: the recoveries are qualitatively comparable with each other and all reasonably
accurate for $\epsilon$ up to $\epsilon=\text{5.00e-2}$. These
observations concur well with the numbers in Table \ref{tab:exam1}.

\begin{table}[hbt!]
\centering
\caption{The reconstruction error $\|q_h^*-q^\dag\|_{L^2(\Omega)}$ for Example \ref{exam:1dsmooth}.\label{tab:exam1}}
\begin{tabular}{c|cccccc}
\hline
$\epsilon$ &  0    & 1.00e-3 &  5.00e-3  & 1.00e-2  & 3.00e-2  & 5.00e-2\\
$\gamma$ &  1.00e-14  & 1.00e-13&  3.00e-13 & 5.00e-13 & 1.00e-12 & 3.00e-12\\
\hline
$\alpha=0.25$ & 7.75e-3 & 9.95e-3 & 1.33e-2 & 1.53e-2 & 2.50e-2 & 3.64e-2\\
$\alpha=0.50$ & 8.73e-3 & 1.00e-2 & 1.33e-2 & 1.50e-2 & 2.65e-2 & 4.11e-2\\
$\alpha=0.75$ & 9.92e-3 & 1.16e-2 & 1.80e-2 & 2.24e-2 & 3.30e-2 & 5.16e-2\\
\hline
\end{tabular}
\end{table}

\begin{figure}[hbt!]
\setlength{\tabcolsep}{0pt}
   \centering
   \begin{tabular}{ccc}
   \includegraphics[width=0.33\textwidth]{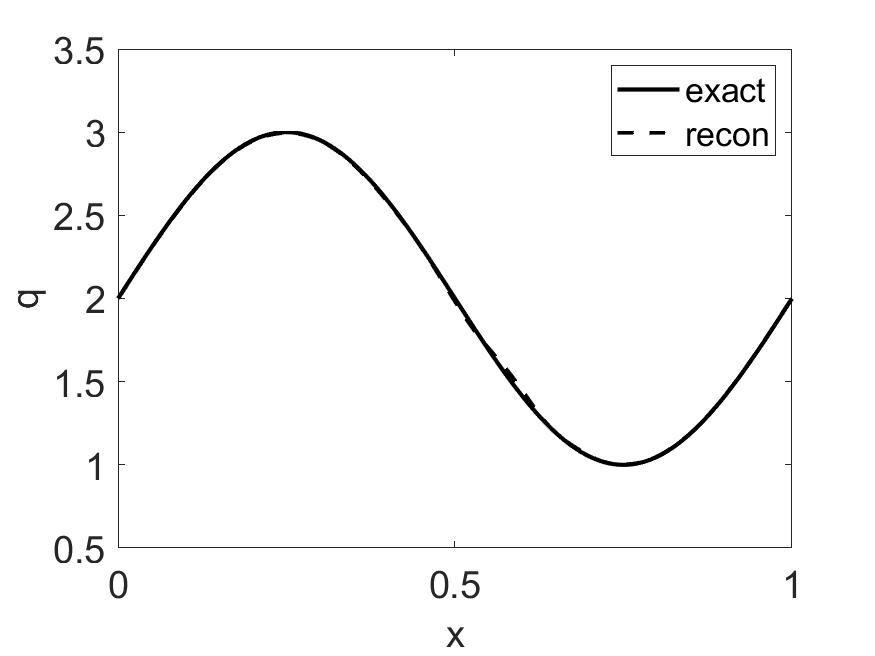}&\includegraphics[width=0.33\textwidth]{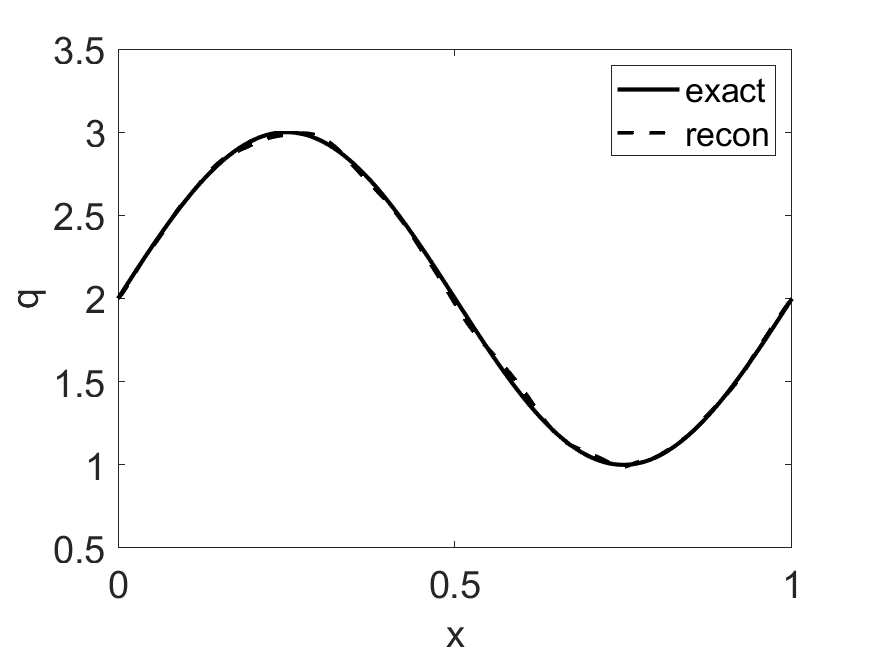} & \includegraphics[width=0.33\textwidth]{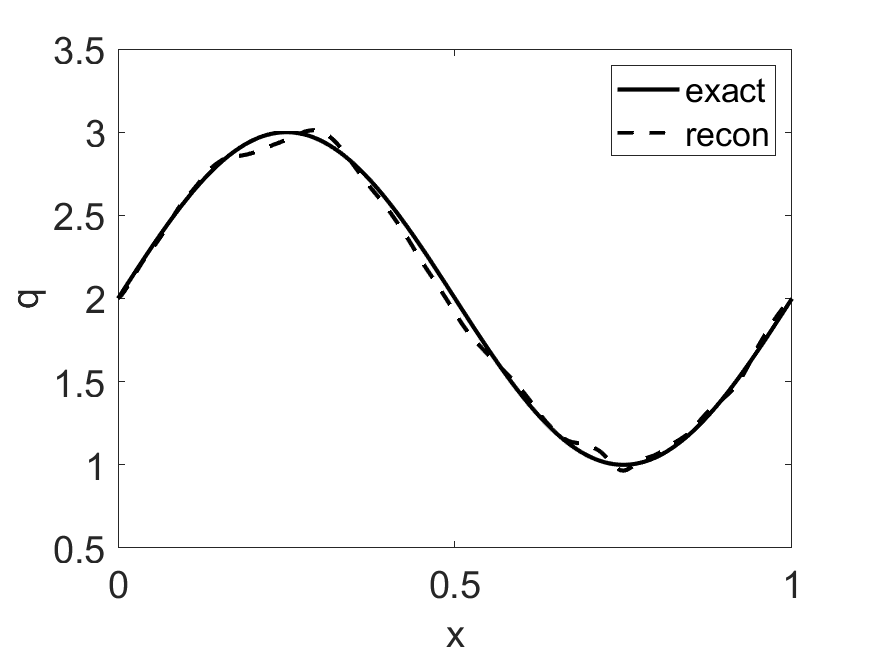}\\
      $\epsilon = 0$  & $\epsilon=\text{1.00e-2}$ & $\epsilon = \text{5.00e-2}$
   \end{tabular}
   \caption{Numerical reconstructions for Example \ref{exam:1dsmooth} with $\alpha=0.5$.
   \label{fig:recon-exam1}}
\end{figure}

Next we examine the convergence with respect to the mesh size $h$ and time step size $\tau$; see
Tables \ref{tab:exam1-conv-h} and \ref{tab:exam1-conv-tau} for the empirical convergence with respect to
$h$ and $\tau$, respectively. The reference regularized solution $q^*$ is computed with $h=1/800$ and $\tau=1/2048$,
and it differs slightly from the exact diffusion coefficient $q^\dag$, due to the presence of data noise ($\epsilon
=\text{1e-2}$). Clearly, the $L^2(\Omega)$ error $\|q^*-q_h^*\|_{L^2(\Omega)}$ of the reconstruction $q_h^*$
(which depends also implicitly on $\tau$ via the optimization problem \eqref{eqn:ob-disc}--\eqref{eqn:fully})
decreases as either the mesh size $h$ or time step size $\tau$ tends to zero, and the convergence is generally
steady. These observations partially confirm the convergence result in \cref{thm:conv-fully}.

\begin{table}[hbt!]
  \centering
  \caption{Reconstruction errors $\|q_h^*-q^*\|_{L^2(\Omega)}$ for Example \ref{exam:1dsmooth} with $\epsilon=\text{1.00e-2}$ (and
  $\beta=\text{5.00e-13}$), v.s. the mesh size $h=1/M$, with $\tau$ fixed at $\tau=2^{-10}$.\label{tab:exam1-conv-h}}
  \begin{tabular}{c|cccccc}
    \hline
    $M$ & 10 & 20 & 40 & 80 & 160 & 320\\
    \hline
 $\alpha=0.25$ & 5.39e-2 & 2.74e-2 & 2.33e-2 & 1.46e-2 & 2.04e-2 & 1.15e-2\\
 $\alpha=0.50$ & 5.38e-2 & 2.56e-2 & 2.51e-2 & 1.56e-2 & 1.16e-2 & 6.51e-3\\
 $\alpha=0.75$ & 4.61e-2 & 2.57e-2 & 2.26e-2 & 2.41e-2 & 1.14e-2 & 8.00e-3\\
    \hline
  \end{tabular}
\end{table}

\begin{table}[hbt!]
  \centering
  \caption{Reconstruction errors $\|q_h^*- q^*\|_{L^2(\Omega)}$ for Example \ref{exam:1dsmooth} with $\epsilon=\text{1.00e-2}$ (and
  $\beta=\text{5.00e-13}$), v.s. the time step size $\tau$, with $h$ fixed at $h=\text{5e-3}$.\label{tab:exam1-conv-tau}}
  \begin{tabular}{c|cccccc}
    \hline
    $\tau$ & $2^{-5}$ & $2^{-6}$ & $2^{-7}$ & $2^{-8}$ & $2^{-9}$ & $2^{-10}$\\
    \hline
 $\alpha=0.25$ & 3.78e-2 & 3.88e-2 & 2.03e-2 & 8.30e-3 & 2.38e-2 & 6.27e-3\\
 $\alpha=0.50$ & 3.90e-2 & 3.80e-2 & 1.98e-2 & 1.92e-2 & 2.07e-2 & 8.46e-3\\
 $\alpha=0.75$ & 9.31e-2 & 4.47e-2 & 2.64e-2 & 1.06e-2 & 1.45e-2 & 6.64e-3\\
    \hline
  \end{tabular}
\end{table}

Last, we take $T_0 = 0$ and examine the convergence of the errors
$ e_q = \| q^\dag - q_h^*  \|_{L^2\II}$ and $e_u= (\tau \sum_{n=1}^N\| u(t_n) - U_h^n(q_h^*) \|_{L^2\II}^2)^{\frac12},  $
with respect to $\epsilon$. Motivated by the error estimates in \cref{thm:error-q}
and Remark \ref{rem:rate}, we fix a small $\tau=1/2048$ and let $h=\sqrt{\epsilon}$ and
$\gamma=10^{-4}\times\epsilon^2$. The errors $e_q$ and $e_u$ are plotted in Fig.
\ref{fig:rate}: a first-order convergence $O(\epsilon)$ is clearly observed. This shows the sub-optimality
of the theoretical convergence rate in \cref{thm:error-q}. This remains an outstanding
question for the analysis of the discrete problem, and seems open even for the
standard parabolic case.

\begin{figure}[htb!]
\centering
\begin{tabular}{cc}
\includegraphics[width=0.40\textwidth]{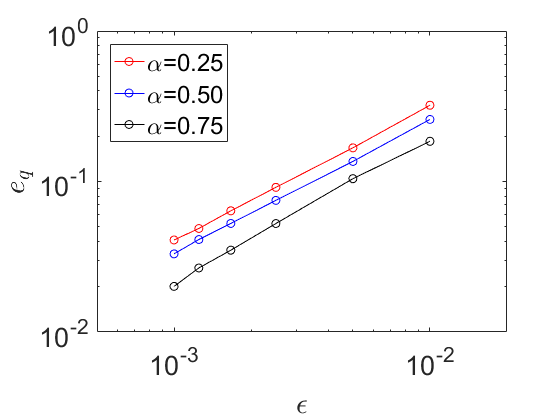}&\includegraphics[width=0.40\textwidth]{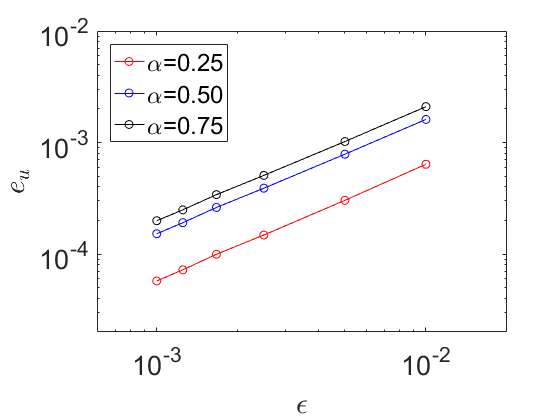}
\end{tabular}
\caption{Plot of $e_u$ and $e_q$ versus $\epsilon$, with $h=\sqrt{\epsilon}$, $\gamma=10^{-4}\times\epsilon^2$ and $\tau=1/2048$.\label{fig:rate}}
\end{figure}

The second example has a nonsmooth exact coefficient $q^\dag$, and the problem is inhomogeneous.
The notation $\min$ denotes the pointwise minimum.
\begin{example}\label{exam:1dnonsmooth}
$u_0(x)=x^2(1-x)^2$, $f(x,t)=e^{x(1-x)}x(1-x)t$, $q^\dag=2+\min(\frac12,\sin^4(2\pi x))$, and $T_0=0.75$.
\end{example}

The numerical results for the example with different noise levels are given in Table \ref{tab:exam2}
and Fig. \ref{fig:recon-exam2}, {where the lower and upper bounds in the admissible set $\mathcal{A}$
are taken to be $c_0=1.9$ and $c_1=2.7$. With this choice, the box
constraint becomes active at some CG iterations.} The observations from Example \ref{exam:1dsmooth} remain largely valid: the
error $e_q=\|q^\dag-q_h^*\|_{L^2(\Omega)}$ decreases as the noise level $\epsilon$ decreases to zero. The results are mostly
comparable for all three fractional orders. For high noise levels, e.g., $\epsilon=\text{5.00e-2}$, the reconstruction error
 is clearly dominated by the oscillations within the flat regions, which is reminiscent of the Gibbs
phenomenon arising from the approximation of the kinks, and also the deviations in the valley. Nonetheless, all the
results are fair and represent acceptable approximations.

\begin{table}[hbt!]
\centering
\caption{Reconstruction error $\|q_h^*-q^\dag\|_{L^2(\Omega)}$ for Example \ref{exam:1dnonsmooth}.\label{tab:exam2}}
\begin{tabular}{c|cccccc}
\hline
$\epsilon$ &  0    & 1.00e-3 &  5.00e-3  & 1.00e-2  & 3.00e-2  & 5.00e-2\\
$\gamma$ &  1.00e-15  & 2.00e-13&  4.00e-13 & 1.00e-12 & 4.00e-12 & 9.00e-12 \\
\hline
$\alpha=0.25$ & 4.36e-3 &  7.91e-3 &  1.28e-2 &  1.56e-2 &  2.21e-2 &  3.02e-2\\
$\alpha=0.50$ & 6.13e-3 &  6.95e-3 &  1.30e-2 &  1.58e-2 &  2.34e-2 &  2.89e-2\\
$\alpha=0.75$ & 1.04e-2 &  1.14e-2 &  1.44e-2 &  1.54e-2 &  2.18e-2 &  3.23e-2\\
\hline
\end{tabular}
\end{table}

\begin{figure}[hbt!]
\setlength{\tabcolsep}{0pt}
   \centering
   \begin{tabular}{ccc}
   \includegraphics[width=0.33\textwidth]{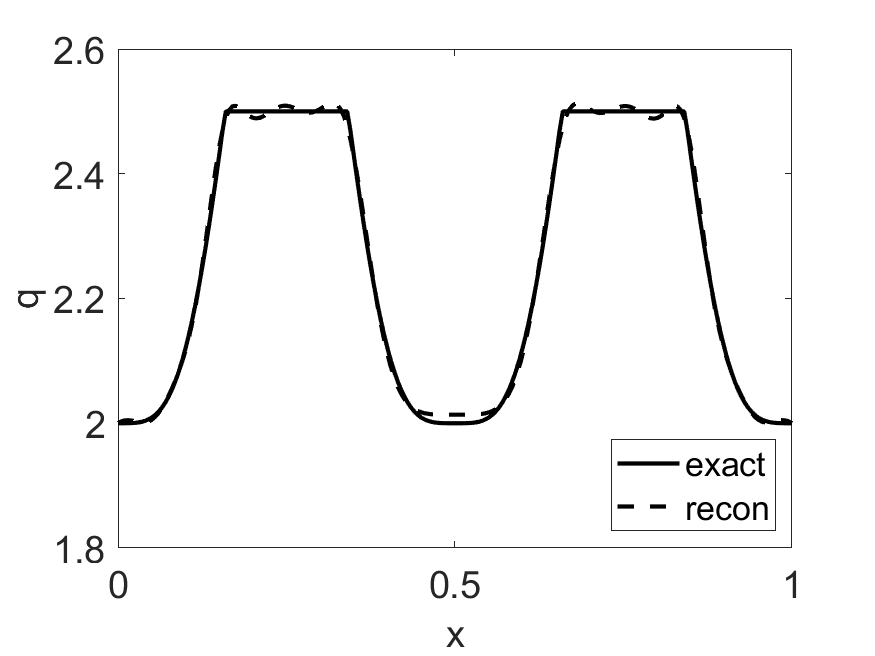}&\includegraphics[width=0.33\textwidth]{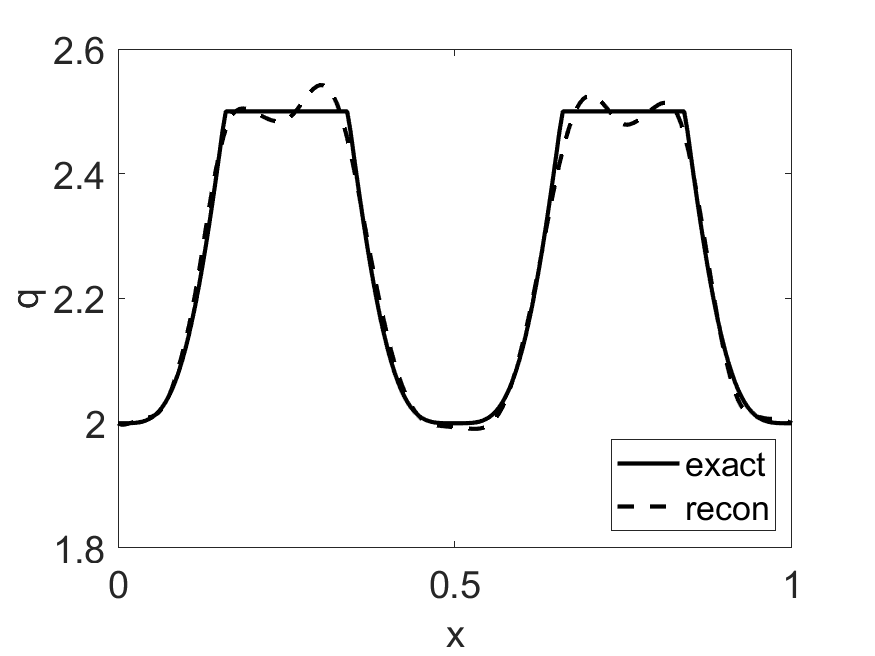} & \includegraphics[width=0.33\textwidth]{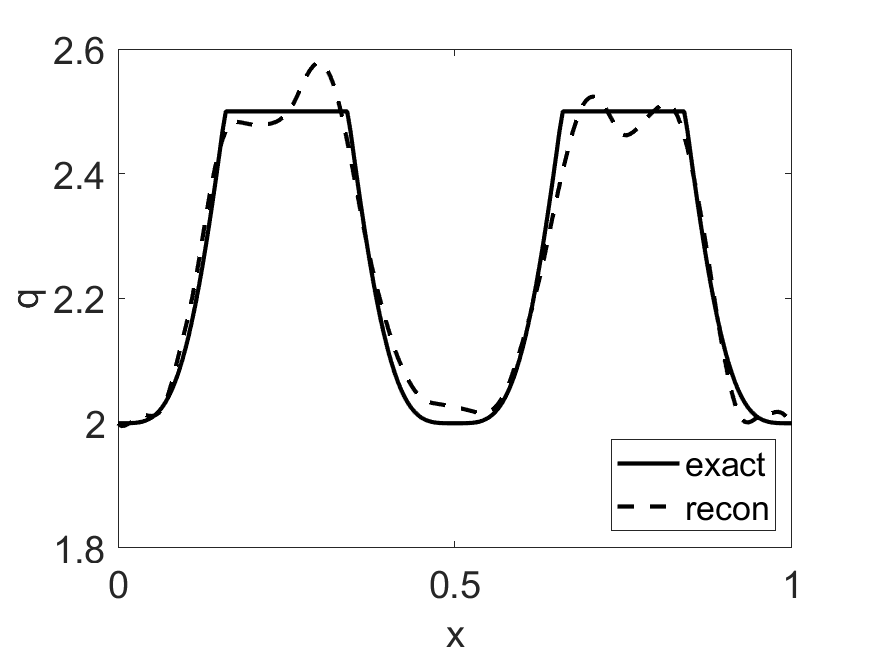}\\
      $\epsilon = 0$  & $\epsilon=\text{1.00e-2}$ & $\epsilon = \text{5.00e-2}$
   \end{tabular}
   \caption{Numerical reconstructions for Example \ref{exam:1dnonsmooth} with $\alpha=0.5$.
   \label{fig:recon-exam2}}
\end{figure}

\subsection{Numerical results in two spatial dimension}

Now we present numerical results for the following example on the unit square $\Omega=(0,1)^2$.
The domain $\Omega$ is first uniformly divided into $M^2$ small squares, each with side length $1/M$, and then a
uniform triangulation is obtained by connecting the low-left and upper-right vertices of each small square.
The reference data is first computed on a finer mesh with $M=100$ and a time step size $\tau=1/2000$.
The inversion is carried out with a mesh $M=40$ and $\tau=1/500$.

\begin{example}\label{exam:2d}
$u_0(x_1,x_2)=x_1(1-x_1)\sin (\pi x_2)$, $f\equiv0$, $q^\dag(x_1,x_2)=1+\sin(\pi x_1) x_2(1-x_2)$, and $T_0=0.8$.
\end{example}

The numerical results for the example with different noise levels are presented in Table
\ref{tab:exam2d} and Fig. \ref{fig:recon-exam3-al50}. The empirical observations
are in excellent agreement with for Example \ref{exam:1dsmooth}, e.g., convergence as the noise level
$\epsilon$ decreases to zero and slightly improved reconstructions for increasing fractional orders $\alpha$. Fig.
\ref{fig:recon-exam3-al50} indicates that the pointwise error $e_q=q_h^*-q^\dag$ lies mainly
in recovering the peak, however, the overall shape is well recovered.

\begin{table}[hbt!]
\centering
\caption{Reconstruction error $\|q_h^*-q^\dag\|_{L^2(\Omega)}$ for Example \ref{exam:2d}.\label{tab:exam2d}}
\begin{tabular}{c|cccccc}
\hline
  $\epsilon$ & 0    & 1.00e-3  & 5e-3 & 1.00e-2  & 3.00e-2 &  5.00e-2 \\
  $\gamma$ & 1.00e-14  & 3.00e-12 & 1.00e-11& 3.00e-11 & 2.00e-10 & 5.00e-10 \\
\hline
 $\alpha=0.25$ & 1.51e-3 & 1.75e-3  & 2.87e-3  &  3.64e-3  &  5.82e-3  & 7.81e-3\\
 $\alpha=0.50$ & 1.61e-3 & 1.86e-3  & 2.80e-3  &  3.62e-3  &  6.58e-3  & 9.57e-3\\
 $\alpha=0.75$ & 1.59e-3 & 2.21e-3  & 3.38e-3  &  4.66e-3  &  1.13e-2  & 1.64e-2\\
 \hline
\end{tabular}
\end{table}

\begin{figure}[hbt!]
\setlength{\tabcolsep}{0pt}
   \centering
   \begin{tabular}{cccc}
   \includegraphics[width=0.25\textwidth]{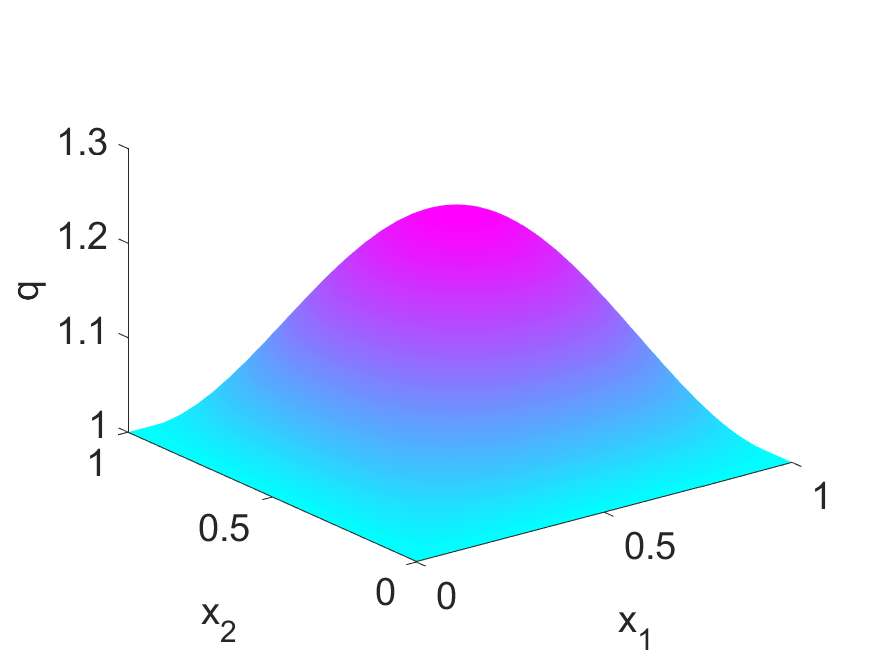}&\includegraphics[width=0.25\textwidth]{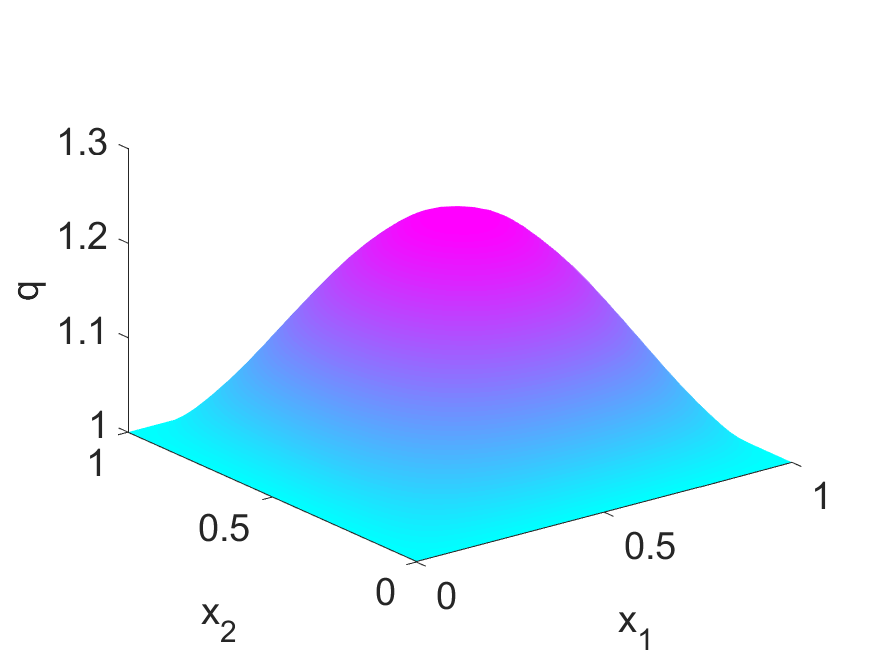} & \includegraphics[width=0.25\textwidth]{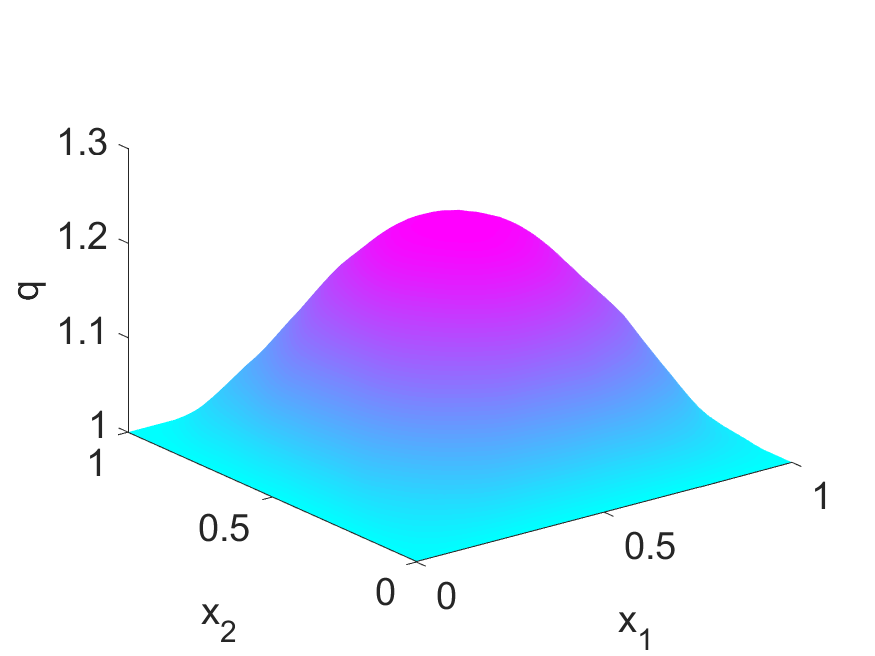} & \includegraphics[width=0.25\textwidth]{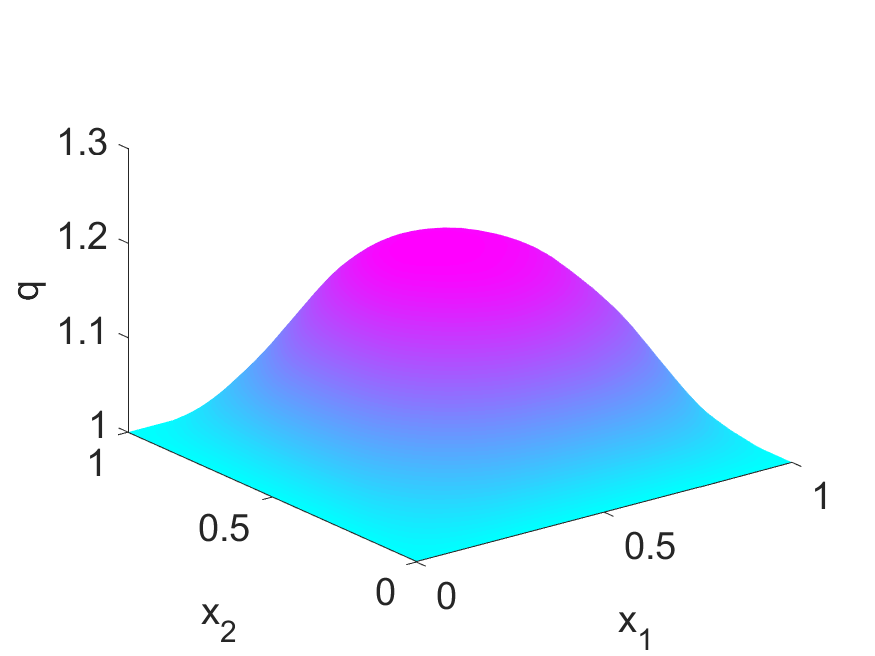}\\
   &\includegraphics[width=0.25\textwidth]{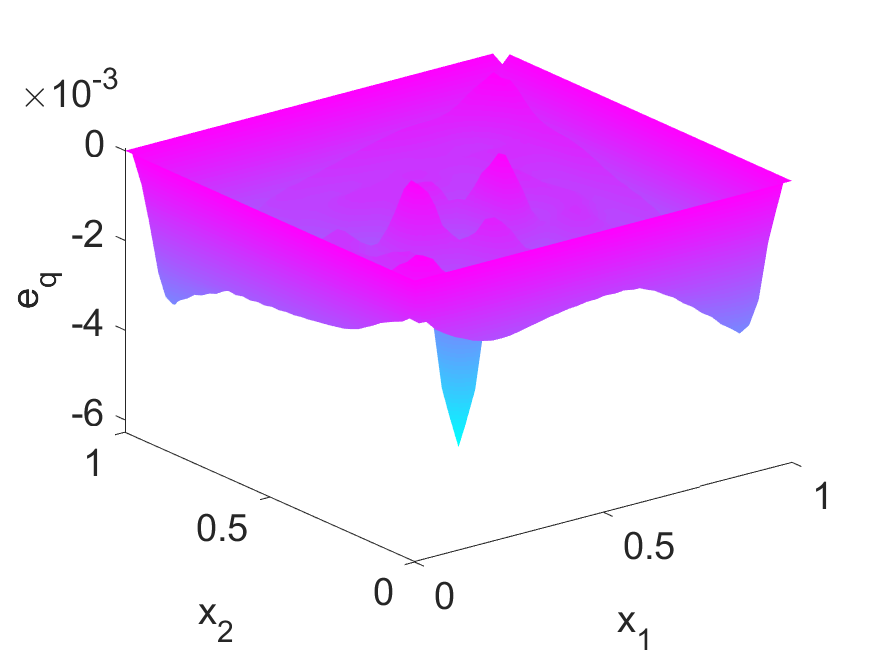} & \includegraphics[width=0.25\textwidth]{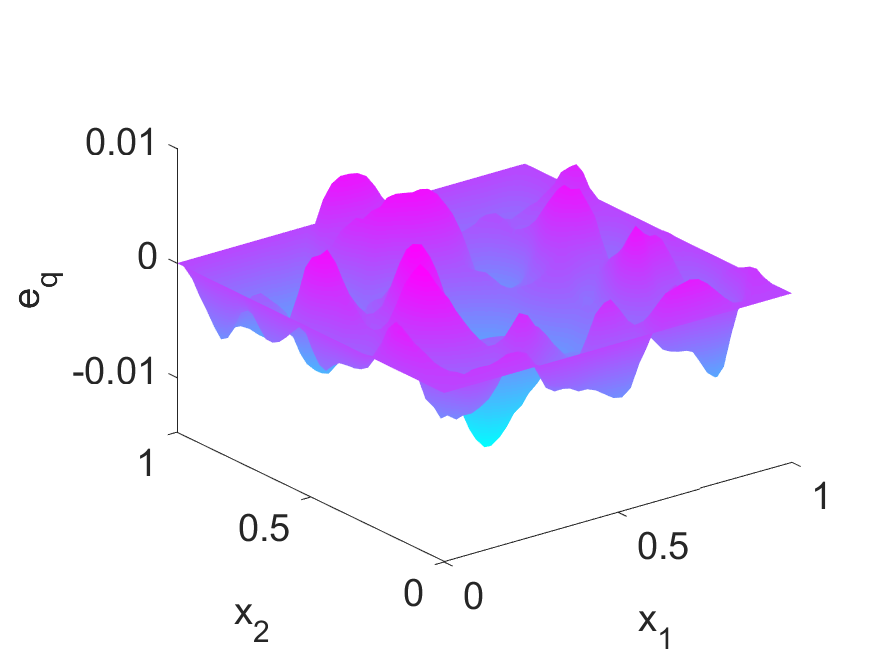} & \includegraphics[width=0.25\textwidth]{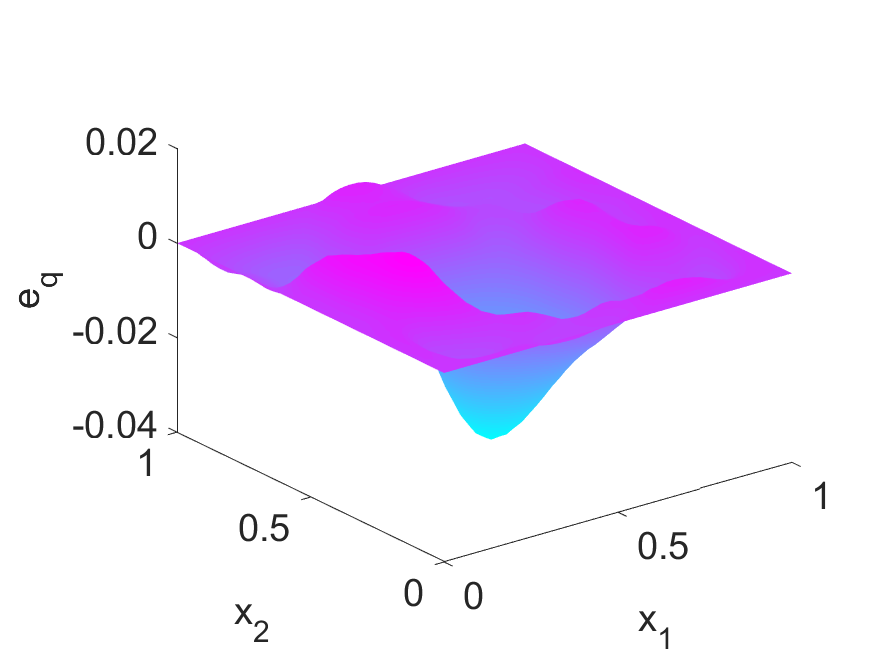}\\
   &   $\epsilon = 0$  & $\epsilon = \text{1.00e-2}$ & $\epsilon = \text{5.00e-2}$
   \end{tabular}
   \caption{Numerical reconstructions for Example \ref{exam:2d} with $\alpha=0.50$.\label{fig:recon-exam3-al50}}
\end{figure}

\section{Conclusions}\label{sec:concl}

In this work, we have studied the numerical recovery of a spatially dependent diffusion
coefficient from the full space-time datum using a regularized least-squares formulation. First, we proved the well-posedness
of the continuous formulation, e.g., stability and convergence. Second, we described a fully discrete scheme based
on the Galerkin finite element method in space and convolution quadrature in time, and showed the convergence of the numerical
approximation. Third, we derived error estimates for the numerical approximation under certain regularity
conditions on the exact diffusion coefficient and problem data.

This work only presents a first step towards rigorous numerical analysis of the inverse conductivity problem.
There are several avenues deserving further research. First, it is important to analyze the formally determined case,
e.g., terminal data or lateral Cauchy data. This is apparently very challenging, since even for the classical
parabolic counterparts, rigorous error estimate (in either a weighted norm or the usual $L^2(\Omega)$) remains elusive.
The techniques in this work also do not extend directly, due to its heavy use of discrete ``integration by parts''
formula over the interval $[0,T]$. Second, even for full data, the obtained error estimates remain suboptimal in
terms of its dependence with the mesh size $h$, when compared with the empirical convergence rate. Partly, this
arises from the inverse inequality, and it remains unclear how to achieve optimality. Third, it is of great interest
to recover the fractional order  $\alpha$ and the diffusion coefficient $q$ simultaneously, or a space-time dependent diffusion
coefficient. {Fourth and last, it is of much interest to derive the necessary and sufficient optimality conditions for the regularized
formulation, to carry out convergence and error analysis with respect to stationary points and to develop more efficient numerical
algorithms. The optimality system may be derived using the spike variation technique in a fairly general setting
(see, e.g., \cite{Lou:2011} for the standard parabolic case).}

\section*{Acknowledgements}
{The authors are grateful to two anonymous referees and the editor, Professor Karl Kunisch, for several
constructive comments that have led to an improvement in the presentation of the paper.}

\bibliographystyle{siam}
\bibliography{frac}

\appendix
\section{Proof of  \cref{lem:deriv-approx}}\label{app:deriv}

The proof relies on the discrete Laplace transform, and the following two estimates
\begin{align}
  \quad c_1 |z| &\le |\delta_\tau(e^{-z\tau})| \le c_2|z| \quad \forall z\in \Gamma_{\theta,\delta}^\tau,\label{eqn:gen}\\
  |\delta_\tau(e^{-z\tau})| &\le |z| \sum_{k=1}^\infty \frac{|z\tau|^{k-1}}{k!} \leq |z|e^{|z|\tau}, \quad \forall z\in \Sigma_\theta,\label{eqn:est-kernel}
\end{align}
with $\Sigma_\theta=\{z\in\mathbb{C}: z\neq0, |\arg(z)|\leq \theta\}$ and $\Gamma_{\theta,\delta}^\tau=\{z=re^{\pm {\rm i}\theta}, \delta \leq r \leq \frac{\pi\sin\theta}{\tau}\}
\cup\{ z= \delta e^{{\rm i}\varphi}: |\varphi|\leq \theta\}$, where $\theta\in (\frac{\pi}{2},\pi)$ is fixed, and the resolvent estimate
\begin{equation}\label{eqn:resol}
  \|(z-A(q))^{-1}\|\leq c|z|^{-1},\quad \forall z\in \Sigma_\theta.
\end{equation}
Now let $y(t)=u(t)-u_0$. Then $y(t)$ satisfies
\begin{equation*}
  \partial_t^\alpha y(t) - Ay(t) + Au_0 = f(t),\quad 0<t\leq T.
\end{equation*}
Taking Laplace transform gives
\begin{equation*}
  z^\alpha \widehat y(z) - A\widehat y(z) + z^{-1}Au_0 = \widehat f(z),
\end{equation*}
i.e., $\widehat y(z) = (z^\alpha-A)^{-1} (\widehat f(z)-z^{-1}Au_0)$. Since $\widehat{\partial_t^\alpha y(t)}=z^\alpha \widehat{y}(z)$
and $\widehat{\bar\partial_\tau^\alpha y}=\delta_\tau(z)^\alpha \widehat{y}(z)$, then $w^n=\partial_t^\alpha y(t_n)-\bar \partial_\tau^\alpha y(t_n)$ is represented by
\begin{align*}
 w^n  & =  \frac{1}{2\pi\mathrm{i}} \int_{\Gamma_{\theta,\delta}^\tau} e^{zt_n} K(z) (z^{-1}A u_0 - \widehat{f}(z))\,\d z  + \frac{1}{2\pi\mathrm{i}} \int_{\Gamma_{\theta,\delta}\setminus\Gamma_{\theta,\delta}^\tau} e^{zt_n} K(z)(z^{-1}A u_0-\widehat{f}(z))\,\d z,
\end{align*}
with $K(z)=(\delta_\tau(e^{-z\tau})^{\alpha}-z^\alpha) (z^\alpha-A)^{-1}$.
Recall the following estimate:
\begin{equation}\label{eqn:gen-1}
  | \delta_\tau(e^{-z\tau})^{\alpha}-z^\alpha| \le c \tau z^{1+\alpha},\quad \forall z\in \Gamma_{\theta,\delta}^\tau.
\end{equation}
By choosing $\delta=c/t_n$ and \eqref{eqn:resol},
${\rm I}=\frac{1}{2\pi\mathrm{i}} \int_{\Gamma_{\theta,\delta}^\tau} e^{zt_n} K(z) z^{-1}(A u_0-f(0))\d z $ is bounded by
\begin{equation*}
\begin{split}
\| {\rm I}\|_{L^2(\Omega)}
   &\le c \tau \| A u_0 -f(0)\|_{L^2(\Omega)} \Big(\int_{\frac{c}{t_n}}^{\frac{\pi\sin\theta}{\tau}} e^{-c\rho t_n} \,\d\rho
   + \int_{-\theta}^\theta ct_n^{-1} \,\d\theta \Big) \\
   &\le c\tau t_n^{-1}  \| A u_0-f(0)\|_{L^2(\Omega)}.
\end{split}
\end{equation*}
Further, by \eqref{eqn:est-kernel}, for any $z=\rho e^{\pm\mathrm{i}\theta}\in \Gamma_{\theta,\delta}\setminus\Gamma_{\theta,\delta}^\tau$ and
choosing $\theta\in(\pi/2,\pi)$ close to $\pi$,
\begin{align*}
  |e^{zt_n} (\delta_\tau(e^{-z\tau})^{\alpha}-z^\alpha)z^{ -1}| & \leq e^{t_n\rho\cos \theta}(c|z|^\alpha e^{\alpha\rho\tau}+|z|^\alpha)|z|^{-1}\leq c|z|^{\alpha-1}e^{-c\rho t_n}.
\end{align*}
Then the term ${\rm II}=\frac{1}{2\pi\mathrm{i}} \int_{\Gamma_{\theta,\delta}\setminus\Gamma_{\theta,\delta}^\tau} e^{zt_n} K(z)z^{-1}(A u_0-f(0))\d z$ is bounded by
\begin{equation*}
\begin{split}
\|{\rm II}\|_{L^2(\Omega)}
   &\le  c \| Au_0-f(0)\|_{L^2(\Omega)}  \int_{\frac{\pi\sin\theta}{\tau}}^\infty e^{-c\rho t_n} \rho^{-1}\,\d\rho
    \le c\tau t_n^{-1} \| Au_0-f(0)\|_{L^2(\Omega)}.
\end{split}
\end{equation*}
In view of the splitting $f(t)=f(0)+ tf'(0) +  {_0I_t^2f''}(t)$, it remains to bound the other two terms. Upon extending $f''(t)$ by
zero to $\mathbb{R}_-$, straightforward computation gives
\begin{align*}
  w^n = -\frac{1}{2\pi\mathrm{i}} \int_{\Gamma_{\theta,\delta}} e^{zt_n} K(z)z^{-2}\,\d zf'(0)\d s
   - \frac{1}{2\pi\mathrm{i}} \int_0^{t_n} \int_{\Gamma_{\theta,\delta}\backslash\Gamma_{\theta,\delta}^\tau} e^{z(t_n-s)}z^{-2}K(z)\,\d zf''(s)\,\d s.
\end{align*}
Then repeating the preceding argument leads to
\begin{equation*}
  \|w^n\|_{L^2(\Omega)} \leq c\tau \Big(\|f'(0)\|_{L^2(\Omega)} + \int_0^{t_n}\|f''(s)\|_{L^2(\Omega)}\d s\Big).
\end{equation*}
Combining the preceding estimates shows the desired assertion.

\end{document}